\documentclass[11pt,reqno]{amsart}
\usepackage{amsmath,amssymb,amscd,amsthm,amsxtra,amsfonts}
\usepackage{mathtools,mathrsfs,dsfont,xparse}
\usepackage{graphicx,epstopdf}
\usepackage{multirow}
\usepackage{cases}
\usepackage{enumerate}
\usepackage[]{hyperref}

\allowdisplaybreaks[3]

\theoremstyle{plain}
\newtheorem{thm}{Theorem}[section]
\newtheorem{prop}[thm]{Proposition}
\newtheorem{lem}[thm]{Lemma}
\newtheorem{cor}[thm]{Corollary}

\theoremstyle{definition}
\newtheorem{defn}[thm]{Definition}
\newtheorem{rmk}[thm]{Remark}

\numberwithin{equation}{section}
\numberwithin{figure}{section}

\newcommand{\NN}{\mathcal{N}}

\newcommand{\C}{\mathbb{C}}

\newcommand{\R}{\mathbb{R}}

\DeclareMathOperator{\re}{Re}
\DeclareMathOperator{\im}{Im}

\DeclareDocumentCommand{\abs}{s m}{
  \operatorname{}
  \IfBooleanTF{#1}{#2}{\left|#2\right|}}

\DeclareDocumentCommand{\norm}{s m}{
  \operatorname{}
  \IfBooleanTF{#1}{#2} {\left\| #2\right\|}}

\DeclareDocumentCommand{\inner}{s m}{
  \operatorname{}
  \IfBooleanTF{#1}{#2}{\left \langle#2\right \rangle}}

\DeclareDocumentCommand{\parenthese}{s m}{
  \operatorname{}
  \IfBooleanTF{#1}{#2}{\left(#2\right)}}

\DeclareDocumentCommand{\square}{s m}{
  \operatorname{}
  \IfBooleanTF{#1} {#2}{\left[#2\right]}}

\DeclareDocumentCommand{\bracket}{s m}{
  \operatorname{}
  \IfBooleanTF{#1}{#2}{\left\{#2\right\}}}

\begin{document}

\title[GWP $\dot{H}^{\frac{1}{2}}$-critical quintic NLS  in 2D]{Global well-posedness and scattering for the defocusing $\dot{H}^{\frac{1}{2}}$-critical nonlinear Schr\"odinger equation in $\R^2$}

\author[Xueying Yu]{Xueying Yu} 
\thanks{The author's current address is Department of Mathematics, MIT, Cambridge, MA.  This work was written at Department of Mathematics, University of Massachusetts Amherst, MA}

\date{}
\begin{abstract}
In this paper we consider the Cauchy initial value problem for the defocusing quintic nonlinear Schr\"odinger equation in two dimensions with general data in the critical space $\dot{H}^{\frac{1}{2}} (\R^2)$. We show that if a solution remains bounded in $\dot{H}^{\frac{1}{2}} (\R^2)$ in its maximal interval of existence, then the interval is infinite and the solution scatters.
\end{abstract}
\maketitle

\section{Introduction}
In the past two decades, there has been significant progress in the understanding of the long time dynamics (existence, uniqueness and scattering) in various regimes of the defocusing nonlinear Schr\"odinger equation (NLS) on $\R^d$
\begin{equation}\label{pNLS}
\begin{cases}
i\partial_t u + \Delta u =  \abs{u}^{p-1} u , \\
u(0, x) = u_0(x)   ,
\end{cases}
\end{equation}
where $u : \R_t \times \R_x^d \to \C$ is a complex-valued function of time and space and $p >1$.

The equation \eqref{pNLS} enjoys several symmetries and invariances among which the most important one is the scaling symmetry. That is, if $u$ solves \eqref{pNLS} with initial data $u_0$, then $u_{\lambda}(t, x)= \lambda^{\frac{2}{p-1}} u(\lambda^2 t , \lambda x)$ solves \eqref{pNLS} with initial data $u_{0,\lambda}(x)= \lambda^{\frac{2}{p-1}} u_0 (\lambda x)$ and $\norm{u_{0,\lambda}}_{\dot{H}^{s}}= \norm{u_0}_{\dot{H}^{s}}$, for any $\lambda \in \R$. The problem becomes scale invariant when the data belongs to the homogenous Sobolev space $\dot{H}^{\frac{d}{2} - \frac{2}{p-1}} (\R^d)$. We refer to this regularity index as {\it critical} relative to scaling and denote it by $s_c:=\frac{d}{2} - \frac{2}{p-1}$.

For data in $H^s(\R^d)$ with $ s> s_c$ (sub-critical regime) local well-posedness (local in time existence, uniqueness and continuous dependence of the data to solution map) follows from the Strichartz estimates and a fixed point argument with a time of existence depending solely on the $H^s$ norm of the data. A similar argument gives also local well-posedness for data in $\dot{H}^{s_c}(\R^d)$ but in this case the time of existence depends also on the profile of the data \cite{CW1, CW2, CW3, Ca}.

The equation \eqref{pNLS} conserves the mass, 
\begin{equation}\label{intro Mass}
M(u(t)) : = \int_{\R^d} \abs{u(t,x)}^2 \, dx = M(u_0) ,
\end{equation} 
and the energy,
\begin{equation}\label{intro Energy}
E(u(t)) :  = \int_{\R^d} \frac{1}{2} \abs{\nabla u(t,x)}^2 + \frac{1}{p+1} \abs{u(t,x)}^{p+1} \, dx = E(u_0) .
\end{equation}
The equation \eqref{pNLS} also conserves momentum,
\begin{equation}\label{intro Momentum}  
\mathcal{P}(u(t)): =  \int_{\R^d} \im [\bar{u}(t,x) \nabla u(t,x)] \, dx = \mathcal{P}(u_0) ,
\end{equation}
which plays a fundamental role in the derivation of the interaction Morawetz estimates \cite{CKSTT1, PV, CGT1, CGT2}. It should be noted however that it does not control the $\dot{H}^{\frac{1}{2}}$ norm globally in time.

In the energy-subcritical regime ($s_c <1$), the conservation of energy gives global well-posedness in $H^1$ by iteration, but this does not give scattering (see Definition \ref{defn Scattering} for details). In the energy-critical regime (see \cite{CW3, Ca}), as we mentioned, the time of existence depends on the profile of the data as well. In this case, if the energy of the initial data is given to be small enough, hence the profile also sufficiently small, then the solution is known to exist globally in time, and scatters.

In the energy-critical regime ($s_c =1$, $p=1 + \frac{4}{d-2}$) with large initial data, one cannot, however, iterate the time of existence to obtain a global solution, because the time of existence depends also on the profile of the data (not a conserved quantity), which may lead to a shrinking interval of existence when iterating the local well-posedness theorem.

Similarly, in the mass-subcritical regime ($s_c < s=0$), the conservation of mass gives global well-posedness in $L^2$ (eg. cubic NLS on $\R$, $s_c= -\frac{1}{2}$) while if $s_c =0$, $p = 1+ \frac{4}{d}$, one cannot extend the local given solution to a global solution by iteration due to the same reason.

The two special cases when the equation becomes scale invariant at the level of one of the conserved quantities \eqref{intro Mass} and \eqref{intro Energy} have received special attention in the past. These are called the mass-critical NLS  ($s _c =0$,\, $p= 1+ \frac{4}{d}$) and the energy-critical NLS ($s _c =1$, \, $p= 1+ \frac{4}{d-2}$). In these two regimes, it is sufficient to prove a uniform {\it a priori} bound for the spacetime $L_{t,x}^{\frac{2(d+2)}{d-2s_c}}$ norm of solutions to the critical NLS, since it is now standard (see \cite{CW2, CW3, Ca}) to show that such a bound gives global well-posedness and scattering for general data.

In the energy-critical case, Bourgain \cite{B1} first introduced an inductive argument on the size of the energy and a refined Morawetz inequality to prove global existence and scattering in three dimensions for large finite energy data which is assumed to be radial. A different proof of the same result is given in \cite{G3}. A key ingredient in the latter proof is an {\it a priori} estimate of the time average of local energy over a parabolic cylinder, which plays a similar role as Bourgain's refined Morawetz inequality. Then, a breakthrough was made by Colliander-Keel-Staffilani-Takaoka-Tao \cite{CKSTT2}. They removed the radial assumption and proved global well-posedness and scattering of the energy-critical problem in three dimensions for general large data. They relied on Bourgain's induction on energy technique to find minimal blow-up solutions that concentrate in both physical and frequency spaces, and proved new interaction Morawetz-type estimates to rule out this kind of minimal blow-up solutions. Later \cite{RV, V} extended the result in \cite{CKSTT2} to higher dimensions.

In \cite{KM1} Kenig and Merle proposed a new methodology, a deep and broad road map to tackle critical problems. In fact, using a contradiction argument they first proved the existence of a critical element such that the global well-posedness and scattering fail. Then relying on a concentration compactness argument they showed that this critical element enjoys a compactness property up to the symmetries of this equation. This final step was reduced to a rigidity theorem that precluded the existence of such critical element. In this form they were able to prove in particular the global well-posedness and scattering for the focusing radially symmetric energy critical Schr\"odinger equation in dimensions three four and five under suitable conditions on the data (Namely, that the energy and the $\dot H^1$ norm of the data are less than those of the ground state). Following their road map Kenig and Merle also showed \cite{KM2} the global well-posedness and scattering for the focusing energy-critical wave equation. It is worth mentioning that the concentration compactness method that they applied was first introduced in the context of Sobolev embeddings in \cite{Ge}, nonlinear wave equations in \cite{BG} and Schr\"odinger equations in \cite{MV, K1, K2}.

The mass-critical global well-posedness and scattering problem was also first studied in the radial case as in \cite{TVZ1, KTV}. Then Dodson proved the global well-posedness of the mass-critical problem in any dimension for nonradial data \cite{D1, D2, D3}. A key ingredient in Dodson's work is to prove a long time Strichartz estimate to rule out minimal blow-up solutions. Such estimate also helps to derive a frequency localized Morawetz-type estimate. Furthermore, in dimension one \cite{D2} and in dimensions two \cite{D3}, Dodson introduced suitable versions of atomic spaces to deal with the failure of the endpoint Strichartz estimates in these cases. It should be noted that the interaction Morawetz estimates proved in \cite{PV} played a fundamental role in ruling out one type of minimal blow-up solutions. Moreover the bilinear estimates in \cite{D3} that gave a logarithmic improvement over Bourgain's bilinear estimteas also relied on the interaction Morawetz estimates of \cite{PV}. For dispersive equations, low dimension settings are less favorable due to the weaker time decay. This paper in fact deals with this  type of setup.

Unlike the energy- and mass-critical problems, for any other $s_c \neq 0,1$, there are no conserved quantities that control the growth in time of the $\dot{H}^{s_c}$ norm of the solutions. In \cite{KM3}, Kenig and Merle showed for the first time that if a solution of the defocusing cubic NLS in three dimensions remains bounded in the critical norm $\dot{H}^{\frac{1}{2}}$ in the maximal time of existence, then the interval of existence is infinite and the solution scatters using concentration compactness and rigidity argument. \cite{Mu2} extended the $\dot{H}^{\frac{1}{2}}$ critical result in \cite{KM3} to dimensions four and higher (some other inter-critical problems were also treated in \cite{Mu1, Mu3}). However, the analogue of the $\dot{H}^{\frac{1}{2}}$-critical result in dimensions two remained open. This was because: 
\begin{enumerate}
\item
the interaction Morawetz estimates in two dimensions are significantly different from those in dimensions three and above,
\item
and the endpoint Strichartz estimate fails.
\end{enumerate}

In this paper, we focus on how to close this gap. More precisely, we consider the Cauchy problem for the defocusing $\dot{H}^{\frac{1}{2}}$-critical quintic Schr\"odinger equation in $\R^{1+2}$:
\begin{equation}\label{NLS}
\begin{cases}
i\partial_t u + \Delta u = \abs{u}^4 u , \\
u(0, x) = u_0(x) \in \dot{H}^{\frac{1}{2}}(\R^{2}),
\end{cases}
\end{equation}
with $u : \R_t \times \R_x^2 \to \C$. 

The main result is:
\begin{thm}[Main theorem] \label{thm Main}
Let  $u : I \times \R^2 \to \C$  be a maximal-lifespan solution to \eqref{NLS} such that $u \in L_t^{\infty}  \dot{H}_x^{\frac{1}{2}} (I \times \R^2 )$. Then $u$ is global and scatters (see Definition \ref{defn Scattering}), with 
\begin{equation}\label{eq Main thm}
\int_{\R} \int_{\R^2} \abs{u(t, x)} ^8 \, dxdt  \leq C \parenthese{\norm{u}_{L_t^{\infty} \dot{H}_x^{\frac{1}{2}}  (\R \times \R^2)} }
\end{equation}
for some function $C : [0, \infty) \to [0,\infty) $.
\end{thm}

In order to prove a uniform {\it a priori} bound for the spacetime $L_{t,x}^{8}$ norm of solutions, following the road map by Kenig and Merle, one proceeds by contradiction as follows: 
\begin{enumerate}
\item[Step 1:]
First assume that the spacetime norm in \eqref{eq Main thm} is unbounded for large data. Then the fact that for sufficiently small initial data the solutions are globally well-posed and their spacetime norms are bounded, imply the existence of a special class of solutions (called {\it minimal blow-up solutions}, see Definition \ref{defn AP}) that are concentrated in both space and frequency.

\item[Step 2:]
One then precludes the existence of minimal blow-up solutions by conservation laws and suitable (frequency localized interaction) Morawetz estimates.
\end{enumerate}

It is shown in \cite{KM3, Mu2} that a minimal blow-up solution $u : I \times \R^d \to \C$ ($d \geq 3$) must concentrate around some spatial center $x(t)$ and at some frequency scale $N(t)$ at any time $t$ in the interval of existence (Step 1). We can extend this to $\R^2$. Then in the step to rule out the possibility of minimal blow-up solutions (Step 2), we employ the effective tools given by Morawetz estimates. For dimensions three and higher, the Morawetz estimates introduced in \cite{LS} are given by:
\begin{equation}\label{intro Morawetz}
\int_I \int_{\R^d} \frac{\abs{u(t,x)}^{p+1}}{\abs{x}} \, dxdt \lesssim \norm{u}_{L_t^{\infty} \dot{H}_x^\frac{1}{2}(I \times \R^d)}^2 .
\end{equation}
\noindent Note that the upper bound on the right-hand side depends only on the $\dot{H}^{\frac{1}{2}}$ norm of the solutions, so it is relatively easy to handle in the $\dot{H}^{\frac{1}{2}}$-critical regime. These were the Morawetz estimates used in \cite{KM3, Mu2}. However in dimensions two, \eqref{intro Morawetz} does not hold. We employ instead the interaction Morawetz estimates, which were first introduced in \cite{CKSTT1} in dimensions three and then extended to dimensions four and higher in \cite{RV}: 
\begin{align*}
- \int_I \iint_{\R^d \times \R^d} \abs{u(t,x)}^2 \Delta\parenthese{\frac{1}{\abs{x-y}}} \abs{u(t,y)}^2 \, dxdydt \\
\lesssim \norm{u}_{L_t^{\infty} L_x^2 (I \times \R^d)}^2 \norm{u}_{L_t^{\infty} \dot{H}_x^\frac{1}{2} (I \times \R^d)}^2 .
\end{align*}

\noindent In dimensions two, the following interaction Morawetz estimate was proved as in \cite{PV, CGT1, CGT2}:
\begin{equation}\label{intro Inter Morawetz}
\norm{ \abs{\nabla}^{\frac{1}{2}} \abs{u(t,x)}^2 }_{L_{t,x}^2 (I \times \R^2)}^2 \lesssim \norm{u}_{L_{t}^{\infty} L_{x}^{2} (I \times \R^2)}^2   \norm{u}_{L_{t}^{\infty} \dot{H}_{x}^{\frac{1}{2}} (I \times \R^2)}^2  . 
\end{equation}
Note that the upper bound above depends on the $\dot{H}^{\frac{1}{2}}$ norm as well as on the $L^2$ norm of the solutions. Hence in our case the right hand side does not need to be finite since there is no {\it a priori} bound on $\norm{u(t)}_{L_x^2}$. In contrast, the upper bounds of the Morawetz estimates \eqref{intro Morawetz} used in \cite{KM3, Mu2} depend only on $\dot{H}^{\frac{1}{2}}$ norm. If we were able to prove an analogue of \eqref{intro Morawetz} in two dimensions (which never holds), we would obtain a bounded spacetime norm immediately and complete Step 2 in the road map by simply adapting the argument in \cite{KM3} to two dimensions. So the difference in the Morawetz estimates raises a headache issue. However, if we truncate the solutions to high frequencies, the right-hand side of \eqref{intro Inter Morawetz} will be finite. As a result, we need to have a good estimate for the error produced in this procedure. Moreover, \eqref{intro Inter Morawetz} scales like $\int_I N(t) \, dt$ (here recall $N(t)$ is the concentration radius of minimal blow-up solutions in the frequency space). Intuitively, the interaction Morawetz estimates are expected to help us to rule out the existence for the solutions satisfying $\int_0^{T_{max}} N(t) \, dt = \infty$.

More precisely, we use the criteria whether $T_{max} (: =\sup I)$ and $\int_0^{T_{max}} N(t) \, dt$ are finite or infinite to classify minimal blow-up solutions ($I$ is the maximal time interval, see Definition \ref{defn Solution}):
\renewcommand\arraystretch{1.5}
\begin{table}[!hbp]
\begin{tabular}{|c|c|c|}
\hline
 &  $T_{max} < \infty$ & $T_{max}= \infty$  \\
\hline
$\int_0^{T_{max}} N(t) \, dt < \infty $  &I & II \\
\hline
$\int_0^{T_{max}} N(t) \, dt  =\infty$  &III & IV \\
\hline
\end{tabular}
\end{table}

\noindent where 
\begin{itemize}
\item
I, III are called finite-time blow-up solutions
\item
I, II are called frequency cascade solutions
\item
III, IV are called quasi-soliton solutions.
\end{itemize}

\noindent In $\dot{H}^{\frac{1}{2}}$ critical regime, it happens that $\int_0^{T_{max}} N(t) \, dt < \infty $ implies $T_{max} < \infty$, hence all frequency cascade solutions are also finite-time blow-up solutions, i.e. there is no {\it Case II} in this setting. Now we proceed to rule out the existence of minimal blow-up solutions case by case.

\noindent {\em Cases III, IV (quasi-soliton solutions)}

These are the cases that we expect the interaction Morawetz estimates will help us to rule out. To deal with these cases, we truncate the solutions to high frequencies, just as it was done in \cite{CKSTT2, RV, V} for the energy-critical problem in dimensions three and above. As a result, we need to derive the more involved good estimate for the low frequency component of the solutions. Here, we recall some ideas and strategies from \cite{D1, D2, D3}.

In the mass-critical problem, Dodson \cite{D1, D2, D3} truncated the solutions to low frequency, since the low frequency component of the solutions was bounded under the $\dot{H}^{\frac{1}{2}}$ norm. (The cutoff in the mass-critical problem and the cutoff in the $\dot{H}^{\frac{1}{2}}$-critical regime are opposite.) To estimate the errors produced by truncating to low frequency, a suitable bound over the high frequency is needed, hence Dodson \cite{D1} introduced the long time Strichartz estimates in dimensions three and higher:
\begin{equation}\label{intro Long time}
\norm{P_{\abs{\xi -\xi(t)} > N} u}_{L_t^2 L_x^{\frac{2d}{d-2}} (J \times \R^d)} \lesssim \parenthese{ \frac{K}{N}}^{\frac{1}{2}} +1,
\end{equation}
where $J$ is an interval satisfying 
\begin{equation*}
\int_J N(t)^3 \, dt = K .
\end{equation*}
Note that $\int_J N(t)^3 \, dt $ scales like \eqref{intro Inter Morawetz} in the mass-critical regime, and it plays the same role as $\int_J N(t) \, dt $ in our setting. 

To obtain \eqref{intro Long time} one requires the endpoint Strichartz estimate $L_t^2 L_x^{\frac{2d}{d-2}}$. For $d \geq 3$ the endpoint estimate  holds, see \cite{KT}, however in dimensions two, as in \cite{T2, T4}, it was shown that the $L_t^2 L_x^{\infty}$ endpoint fails. This failure causes the great difficulty for defining and proving the long time Strichartz estimates. To conquer this defect, Dodson \cite{D3} constructed a new function space out of a certain atomic spaces in dimensions two. This construction captured the essential features of the long time Strichartz estimates \eqref{intro Long time} in dimensions three and higher.

Back to our case, we follow Dodson's idea in \cite{D3} and construct a similar but `upside-down' version of a long time Strichartz estimate adapted to the $\dot{H}^{\frac{1}{2}}$-critical setting. Our `upside-down' method is because in the mass-critical regime, Dodson \cite{D3} lost the {\it a priori} control in the $\dot{H}^{\frac{1}{2}}$-norm and used the long time Strichartz to quantify how bad $\dot{H}^{\frac{1}{2}}$-norm is out of control, while we lose the {\it a priori} control of $L^2$-norm. Hence we define a long time Strichartz estimate over low frequency and expect that it gives a good control of the low frequency components of the solutions. Moreover, in our proof of the long time Strichartz estimate, we should be very careful with the high frequency and high frequency interaction into low frequency terms in Littlewood-Paley decomposition. It is because such terms require more summability due to the construction of the atomic spaces where the long time Strichartz estimate lives. This forces us to gain more decay than the mass-critical case to sum over the high frequency terms. In contrast, these terms were not problematic in mass-critical \cite{D3}, because with the opposite cutoff the worse case was all low frequencies interaction into high frequency. However, this case never happens since the contribution of all low frequencies remains low. With the error terms settled, the frequency localized Morawetz estimate will help us to preclude the existence of the quasi-soliton solutions.

\noindent {\em Case I (finite-time blow-up solutions)}

Next, we are left to rule out {\it Case I}. In $\dot{H}^{s_c}$ regime ($s_c \neq \frac{1}{2}$), see \cite{KV3, D1, D2, D3, Mu1}, long time Strichartz estimates help to prove either additional decay or additional regularity for the solutions belonging to {\it Cases I, II}, as a result, the solutions in these cases can be shown to have zero mass or energy, which contradicts the fact they are blow-up solutions. However, in the $\dot{H}^{\frac{1}{2}}$ critical regime, due to the scaling, the long time Strichartz estimates do not provide any additional decay as we would like. This forces us to treat {\it Case I} as finite-time blow-up solutions, instead of as frequency cascade solutions. Since we do not need the additional information from $\int N(t)\, dt$, we are allowed to consider the finite-time blow-up solutions ({\it Cases I, III}) together. In these cases, by considering the rate of change (in time) of the mass of solutions restricted within a spacial bump and using the finiteness of blow-up time, we can see the impossibility of this type of minimal blow-up solutions.

\subsection{A more formal outline of the proof}
We first define
\begin{defn}[Solution]\label{defn Solution}
A function $u : I \times  \R^2 \to C$ on a time interval $I (\ni 0)$ is a solution to \eqref{NLS} if it belongs to $C_t \dot{H}_x^{\frac{1}{2}} (K \times \R^2 ) \cap L_{t,x}^8 (K \times \R^2 ) $ for every compact $K \subset I$ and obeys the Duhamel formula 
\begin{equation*}
u(t)= e^{it\Delta}u_0 -i\int_0^t   e^{i(t-t^{\prime})\Delta}  \parenthese{ \abs{u}^4 u} (t^{\prime},x) \, dt^{\prime}
\end{equation*}  
for all $t \in I$. We call $ I$ the lifespan of $u$; we say $u$ is a maximal-lifespan solution if it cannot be extended to any strictly larger interval. If $I = R$, we say $u$ is global.
\end{defn}

\begin{defn}[Scattering]\label{defn Scattering}
A solution to \eqref{NLS} is said to scatter forward (or backward) in time if there exist $u_{\pm} \in \dot{H}^{\frac{1}{2}} (\R^2)$ such that
\begin{equation*}
\lim\limits_{t \to \pm \infty}  \norm{u(t) -e^{it \Delta} u_{\pm}}_{\dot{H}_x^{\frac{1}{2}}(\R^2)}=0.
\end{equation*}
\end{defn}

\begin{defn}[Scattering size and blow up]
We define the scattering size of a solution $u$ to \eqref{NLS} on a time interval $I$ by
\begin{equation*}
S_I(u) =\int_I \int_{\R^2} \abs{u(t, x)} ^8 \, dxdt .
\end{equation*}

If there exists $t \in I$ such that $S_{[t, \sup I )} (u) = \infty$, then we say $u$ blows up forward in time. Similarly, if there exists $t \in I$ such that $S_{(\inf I,t]}(u) = \infty$, then we say $u$ blows up backward in time.
\end{defn}

The local theory for \eqref{NLS} has been worked out by Cazenave-Weissler \cite{CW1, CW2, CW3, Ca}.

\begin{thm}[Local well-posedness]
Assume $u_0 \in \dot{H}^{\frac{1}{2}}(\R^2)$. Then there exists a unique maximal-lifespan solution $u: I \times \R^2 \to \C$ to \eqref{NLS}  such that: 
\begin{enumerate}[(i)]
\item 
(Local existence) $I$ is an open interval containing $0$.

\item
(Blowup criterion) If $\sup I $ is finite, then the solution $u$ blows up forward in time. If $\inf I $ is finite, then the solution $u$ blows up backward in time.

\item
(Scattering) If $u$ does not blow up forward in time then $\sup I =\infty$ and u scatters forward in time. If $u$ does not blow up backward in time then $\inf I =\infty$ and u scatters backward in time. 

\item
(Small-data global existence) If $\norm{ u_0}_{  \dot{H}_x^{\frac{1}{2}}}$ is sufficiently small then the solution $u$ is global, scatters and does not blow up either forward or backward in time, with $S_{\R} (u) \lesssim \norm{u_0}_{  \dot{H}_x^{\frac{1}{2}}}^{8}$

\end{enumerate}
\end{thm}

To prove Theorem \ref{thm Main}, we, following the road map in \cite{KM1}, argue by contradiction. First, define $L: [0, \infty) \to [0, \infty]$ by 
\begin{equation*}
L(E) := \sup \bracket{ S_I(u) \big|  u: I \times \R^2 \to \C   \text{ solving }  \eqref{NLS}  \text{ with }  \norm{u}_{L_{t}^{\infty} \dot{H}_x^{\frac{1}{2}} (I \times \R^2) }^2  \leq E } .
\end{equation*}

\noindent Note that $L(E)$ is a non-negative, non-decreasing and continuous function satisfying $L(E) \leq E^4 $ for $E$ sufficiently small. Then there must exist a unique critical threshold $E_c \in (0, \infty]$ such that 
\begin{equation*}
L(E)
\begin{cases}
< \infty  \quad \text{ if } E<E_c \\
=\infty \quad \text{ if } E\geq E_c .
\end{cases}
\end{equation*}
The failure of main theorem implies that $0 < E_c < \infty$. Moreover, 
\begin{thm}[Existence of minimal counterexamples]\label{thm Exist minimal}
Suppose Theorem \ref{thm Main} fails to be true. Then there exists a critical energy $0 < E_c <\infty$ and a maximal-lifespan solution $u : I  \times \R^2 \to  \C$ to \eqref{NLS} with $\norm{u}_{L_{t}^{\infty} \dot{H}_x^{\frac{1}{2}} (I \times \R^2) }^2 = E_c$, which blows up in both time directions in the sense that
\begin{equation*}
S_{\geq 0 } (u) = S_{\leq 0} (u)= \infty
\end{equation*}
and whose orbit $\{u(t) : t \in \R \}$ is precompact in $\dot{H}^{\frac{1}{2}}$ modulo scaling and spatial translations.
\end{thm}

The proof of Themreom \ref{thm Exist minimal} follows from Proposition 3.3 and Proposition 3.4 in \cite{KM3}. See also the proof of Theorem 1.7 in \cite{Mu2}. The key ingredient in the proof is a profile decomposition argument (see Lemma \ref{lem Profile decomp}).

The maximal-lifespan solution found in Theorem \ref{thm Exist minimal} is almost periodic modulo symmetries:
\begin{defn}[Almost periodicity]\label{defn AP}
A solution $u$ to \eqref{NLS} with lifespan $I$ is said to be almost periodic (modulo symmetries) if $u \in L_{t}^{\infty} \dot{H}_x^{\frac{1}{2}} (I \times \R^2)$ and there exist (possibly discontinuous) functions: $N: I \to \R^+$, $x: I \to \R^2$, $C: \R^+ \to \R^+$ such that:
\begin{equation}\label{eq Defn AP1}
\int_{\abs{x-x(t)} \geq \frac{C(\eta)}{N(t)}}  \abs{\abs{\nabla}^{\frac{1}{2}} u(t,x)}^2 \, dx + \int_{\abs{\xi} \geq C(\eta)N(t)}  \abs{\xi} \abs{\hat{u}(t,\xi)}^2  \, d\xi \leq \eta
\end{equation}
for all $t \in I $ and $\eta > 0$. We refer to the function $N$
 as the frequency scale function, $x$ is the spatial center function, and $C$ is the compactness modulus function.
\end{defn}
Notice that the Galilean transformation only preserves the $L^2$ norm of $u$, not the $\dot{H}^{s_c}$  norm where  $s_c > 0 $. Hence we have no Galilean transformation in our case and the frequency center $\xi (t)$ is the origin. We will use this later in Definition \ref{defn X}.

Another consequence of the precompactness in $\dot{H}^{\frac{1}{2}}$ modulo symmetries of the orbit of the solution found in Theorem \ref{thm Exist minimal} is that for every $\eta > 0$ there exists $c(\eta) >0$ such that
\begin{equation}\label{eq Defn AP2}
\int_{\abs{x-x(t)} \leq \frac{c(\eta)}{N(t)}}  \abs{\abs{\nabla}^{\frac{1}{2}} u(t,x)}^2 \, dx + \int_{\abs{\xi} \leq c(\eta)N(t)}  \abs{\xi} \abs{\hat{u}(t,\xi)}^2 \, d\xi \leq \eta
\end{equation}
uniformly for all $t \in I$.

For nonnegative quantities $X$ and $Y$, we write $X \lesssim Y$ to denote the inequality $X \leq CY$ for some constant $C >0$. If $X \lesssim Y \lesssim X$, we write $X \sim Y$. The dependence of implicit constants on parameters will be indicated by subscripts, for example, $X \lesssim_u Y$ denotes $X \leq CY$ for some $C = C(u)$. 

For these almost period solutions, $N(t)$ enjoys the following properties: Lemma \ref{lem Local const N(t)}, Corollary \ref{cor N(t) blow-up}, Lemma \ref{lem Local bounded} and Lemma \ref{lem Strichartz norms N(t)} (see \cite{KV1} for details):

\begin{lem}[Local constancy of $N(t)$ and $x(t)$]\label{lem Local const N(t)}
Let $u : I \times \R^2 \to \C$ be a non-zero almost periodic modulo symmetries solution to \eqref{NLS} with parameters $N(t)$ and $x(t)$. Then there exists a small number $\delta$, depending on $u$, such that for every $t_0 \in I$ we have
\begin{equation*}
[t_0 - \delta N(t_0)^{-2}, t_0 + \delta N(t_0)^{-2}] \subset I
\end{equation*}
and
\begin{equation*}
N(t) \sim N(t_0) \text{ and } \abs{x(t) - x(t_0)} \lesssim N(t_0)^{-1}
\end{equation*}
whenever $\abs{t-t_0} \leq \delta N(t_0)^{-2}$.
\end{lem}

\begin{rmk}\label{rmk N(J)}
If $J$ is an interval with
\begin{equation*}
\norm{u}_{L_t^8 L_x^8 (J \times \R^2)} =1,
\end{equation*}
then for $t_1, t_2 \in J$, 
\begin{equation*}
N(t_1) \sim N(t_2) .
\end{equation*}
Then combine with Lemma \ref{lem Local const N(t)}, we can choose $N(t)$ such that
\begin{equation*}
\abs{\frac{d}{dt}N(t) }  \lesssim N(t)^3   \parenthese{\implies  \abs{ \frac{d}{dt} \parenthese{\frac{1}{N(t)} }}  \lesssim N(t) } .
\end{equation*} 
Define 
\begin{equation*}
N(J) = \inf_{t \in J } N(t),
\end{equation*}
then
\begin{equation*}
\frac{1}{N(J)} \sim \int_J N(t) \, dt  .
\end{equation*}
\end{rmk}

\begin{cor}[$N(t)$ at blow-up]\label{cor N(t) blow-up}
Let $u : I \times \R^2 \to \C$ be a non-zero maximal-lifespan solution to \eqref{NLS} that is almost periodic modulo symmetries with frequency scale function $N : I \to \R^+$. If $T$ is any finite endpoint of the lifespan $I$, then $N(t)\gtrsim   \abs{T -t}^{-\frac{1}{2}}$; in particular, $\lim_{t \to T} N(t) = \infty$. If $ I$ is infinite or semi-infinite, then for any $t_0 \in I$ we have $N (t) \gtrsim  \min \{N (t_0 ), \abs{t -t_0}^{-\frac{1}{2} } \}$.
\end{cor}

\begin{lem}[Local quasi-boundedness of $N(t)$]\label{lem Local bounded}
Let $u$ be a non-zero solution to \eqref{NLS} with lifespan $I$ that is almost periodic modulo symmetries with frequency scale function $N: I \to \R^+$. If $K$ is any compact subset of $I$, then
\begin{equation*}
0 < \inf_{t \in K} N(t) \leq \sup_{t \in K} N(t) < \infty  .
\end{equation*}

\end{lem}

\begin{lem}[Strichartz norms via $N(t)$]\label{lem Strichartz norms N(t)}
Let $u : I \times \R^2 \to \C$ be a non-zero almost periodic modulo symmetries solution to \eqref{NLS} with frequency scale function $N : I \to  \R^+$. Then
\begin{equation*}
\int_I N(t)^2 \, dt \lesssim \int_I \int_{\R^2} \abs{u(t,x)}^8 dxdt \lesssim 1+ \int_I N(t)^2 \,  dt .
\end{equation*}

\end{lem}

With the above setup and properties in hand, we arrive at the following theorem:

\begin{thm}[two special scenarios for blow-up]\label{thm 2 types}
Suppose Theorem \ref{thm Main} failed. Then there exists an almost periodic solution $u : [0, T_{max}) \times   \R^2   \to \C $, such that \eqref{eq Defn AP1}, \eqref{eq Defn AP2},
\begin{equation*}
\norm{u}_{L_{t,x}^8 ( [0, T_{max}) \times   \R^2)} = + \infty,
\end{equation*}
$N(0)=1$, and $N(t) \geq 1$ on $[0, \infty)$, $\abs{\frac{d}{dt} N(t)} \lesssim N(t)^3$.

Furthermore, one of the following holds:
\begin{enumerate}
\item The finite-time blow-up solutions,
\begin{equation*}
T_{max} < \infty ,
\end{equation*}

\item The quasi-soliton,
\begin{equation*}
\int_0^{\infty} N(t)  \,  dt =\infty .
\end{equation*}
\end{enumerate}

\end{thm}

For finite time blow-up solutions, we compute the rate of change in time of the mass restricted to a spatial bump and use the mass conservation law to rule out the existence.

To preclude the second case, we follow the following steps: 
\begin{enumerate}
\item
Step 1: Prove a suitable long time Strichartz estimate: $\norm{u}_{\tilde{X}_{k_0}} \lesssim 1$,

\item
Step 2: Derive a frequency localized Morawetz estimate with error terms estimated by the long time Strichartz estimate,

\item
Step 3: Using the frequency localized Morawetz estimate, rule out the quasi-soliton solutions (recall that Morawetz inequality scales like $\int_I N(t)  \,  dt$).
\end{enumerate}

The rest of this paper is organized as follows:
In Section \ref{sec Preliminaries}, we collect some useful tools in harmonic analysis and a profile decomposition argument, and recall the local well-posedness theory and a perturbation lemma. In Section \ref{sec No finite-time blow-up}, we show the impossibility of finite-time blow-up solutions. Next, in Sections \ref{sec Atomic spaces}, we review some basic definitions and properties of the atomic spaces, and then prove a decomposition lemma, which is used in the proof of the long time Strichartz estimate in Section \ref{sec LTS}. In Section \ref{sec LTS}, we derive a long time Strichartz estimate adapted in our setting. Finally, we prove the frequency-localized interaction Morawetz estimates, then rule out the existence of quasi-soliton solutions in Section \ref{sec No quasi-soliton}, which completes the proof of Theorem \ref{thm 2 types}.

\section{Preliminaries}\label{sec Preliminaries}
In this section we recall Littlewood-Paley theory, some useful estimates from harmonic analysis and a profile decomposition argument, and state the local well-posedness theory and a perturbation lemma.

\subsection{Littlewood-Paley theory}
\begin{lem}[Littlewood-Paley theorem]
For $1 < p < \infty$, 
\begin{equation*}
\norm{f}_{L^p(\R^2)} \sim_p \norm{ \parenthese{\sum_{j=-\infty}^{\infty} \abs{P_{2^j} f}^2}^{\frac{1}{2}}}_{L^p(\R^2)} .
\end{equation*}
\end{lem}

\begin{lem}[Bernstein inequalities]
For $1 \leq r \leq q \leq \infty$ and $s \geq 0$,
\begin{align*}
\norm{\abs{\nabla}^{\pm s} P_N f}_{L_x^r (\R^2)} & \lesssim N^{\pm s} \norm{ P_N f}_{L_x^r (\R^2)} \\
\norm{\abs{\nabla}^{ s} P_{\leq N} f}_{L_x^r (\R^2)} & \lesssim N^{ s} \norm{ P_{\leq N} f}_{L_x^r (\R^2)} \\
\norm{ P_{\geq N} f}_{L_x^r (\R^2)} & \lesssim N^{- s} \norm{ \abs{\nabla}^{s} P_{\geq N} f}_{L_x^r (\R^2)} \\
\norm{ P_{\leq N} f}_{L_x^q (\R^2)} & \lesssim N^{\frac{2}{r}-\frac{2}{q}} \norm{  P_{\leq N} f}_{L_x^r (\R^2)} .
\end{align*}

\end{lem}

\subsection{Estimates from harmonic analysis}
Next we recall some fractional calculus estimates that appear originally in \cite{CW}. For a textbook treatment, one can refer to \cite{Ta}.
\begin{lem}[Fractional product rule]
Let $s > 0$ and let $1 < 1, r_1, r_2, q_1, q_2 < \infty$ satisfy $\frac{1}{q}=\frac{1}{r_1} + \frac{1}{r_2}$, $\frac{1}{q} = \frac{1}{p_1} + \frac{1}{p_2}$. Then
\begin{equation*}
\norm{\abs{\nabla}^s (fg)}_{L_x^q} \lesssim \norm{f}_{L_x^{r_1}} \norm{\abs{\nabla}^s g}_{L_x^{r_2}} + \norm{\abs{\nabla}^s f}_{L_x^{p_1}} \norm{ g}_{L_x^{p_2}}  . 
\end{equation*}
\end{lem}

\begin{lem}[Chain rule for fractional derivatives]\label{lem Chain rule}
If $F \in C^{2}$, with $F(0)=0$, $F^{\prime}=0$, and $\abs{F^{\prime \prime}(a+b)} \leq C \square{ \abs{F^{\prime \prime}(a)}+\abs{F^{\prime \prime}(b)} }$, and $\abs{F^{\prime}(a+b)} \leq \square{ \abs{F^{ \prime}(a)}+\abs{F^{\prime}(b)} }$, we have, for $0 < \alpha <1$,
\begin{equation*}
\norm{\abs{\nabla}^{\alpha} F(u)}_{L_{x}^{q}} \leq C \norm{F^{\prime}(u)}_{L_{x}^{p_1}}  \norm{\abs{\nabla}^{\alpha} u}_{L_{x}^{p_2}}, \quad \text{where } \frac{1}{q}= \frac{1}{p_1}+ \frac{1}{p_2}. 
\end{equation*}

\end{lem}

\begin{lem} [Sobolev embedding]
For $\forall v \in C^{\infty}_{0} (\R ^d)$, $\frac{1}{p}-\frac{1}{q}=\frac{s}{d}$ and $s>0$, we have 
\begin{equation*}
\norm{  v }_{L^{q}_{x}(\R^d)} \leq C \norm{\abs{\nabla}^{s} v}_{L^{p}_{x}(\R^d)},  \text{ i.e. }  \dot{W}^{s, p}(\R^d) \hookrightarrow L^{q}(\R^d)   .
\end{equation*}
\end{lem}

\begin{lem}[Gagliardo-Nirenberg interpolation inequality]\label{lem GN ineq}
Let $1 < p < q \leq \infty$ and $s>0$ be such that $\frac{1}{q}= \frac{1}{p}- \frac{s \theta}{d}$ for some $0< \theta=\theta(d,p,q,s) < 1$. Then for any $u \in \dot{W}^{s,p} (\R^d) $, we have 
\begin{equation*}
\norm{u}_{L^q (\R^d)} \lesssim_{d,p,q,s} \norm{u}_{L ^p(\R^d)}^{1-\theta} \norm{u}_{\dot{W}^{s,p}(\R^d)}^{\theta} .
\end{equation*} 
	
\end{lem}

\begin{lem}[Hardy's inequality]
For any $0 \leq s< d/2$, there exists $c=c(s,d)>0$, such that 
\begin{equation*}
\norm{\frac{f(x )}{\abs{x}^s}}_{L^2(\R^d)} \leq c(s,d) \norm{f}_{\dot{H}^s (\R^d)} .
\end{equation*}

\end{lem}


\begin{lem}[Hardy-Littlewood-Sobolev inequality]
Let $0 < r< d$, $1 < p < q< \infty $, $\frac{1}{p} -\frac{1}{q} =1-\frac{\gamma}{d}$. then if $R_{\gamma}(x) =\frac{1}{\abs{x}^{\gamma}}$,
\begin{equation*}
\norm{ \abs{ \,  \cdot  \, }^{- \gamma}  \ast f}_{L^q(\R^d)} = \norm{R_{\gamma} \ast f}_{L^q(\R^d)} \lesssim \norm{f}_{L^p (\R^d)} .
\end{equation*}
\end{lem}

\subsection{Strichartz estimates}
First, we define the following spaces:
\begin{defn}[Strichartz spaces]
We call a pair of exponents $(q,r)$ admissible if 
\begin{equation*}
\frac{1}{q} + \frac{1}{r} =\frac{1}{2}
\end{equation*}
for all $2 < q \leq \infty$, $2 \leq r < \infty$. Notice that we do not have the endpoint in two dimensions, i.e. $(q,r) \neq (2, \infty)$.

For an interval $I$ and $s \geq 0$, we define the Strichartz norm by
\begin{equation*}
\norm{u}_{\dot{S}^s (I)} : = \sup \bracket{ \norm{\abs{\nabla }^s u}_{L_t^q L_x^r (I \times \R^2)} : (q,r) \text{ is admissible}}.
\end{equation*}
\end{defn}

For a time interval $I$, we write $L_t^q L_x^r (I \times \R^2)$ for the Banach space of functions $u : I \times \R^2 \to \C$ equipped with the norm
\begin{equation*}
\norm{u}_{L_t^q L_x^r (I \times \R^2)} : = \parenthese{\int_I \norm{u(t)}_{L_x^r (\R^2)}^q \, dt}^{\frac{1}{q}} .
\end{equation*}
For $1 \leq r \leq \infty$, we denote the conjugate or dual exponent of $r$ by $r^{\prime}$, i.e. $r^{\prime}$ satisfies $\frac{1}{r} + \frac{1}{r^{\prime}} =1$.

\begin{lem}[Strichartz estimates]
For any admissible exponents $(q,r)$ and $(\tilde{q}, \tilde{r})$ we have Strichartz estimates:
\begin{align*}
&\norm{e^{it\Delta} u_0}_{L_t^q  L_x^r  (\R \times \R^2)} \lesssim_{q,r} \norm{u_0}_{L_x^2 (\R^2)}\\
&\norm{\int_{\R} e^{-is\Delta} F(s) \, ds }_{L_x^2 (\R^2)} \lesssim_{\tilde{q}, \tilde{r}} \norm{F}_{L_t^{\tilde{q}^{\prime} } L_x^{\tilde{r}^{\prime}}   (\R \times \R^2)}\\
&\norm{\int_{t^{\prime} < t} e^{i(t-t^{\prime})\Delta} F(t^{\prime}) \, dt^{\prime} }_{L_t^q  L_x^r(\R \times \R^2)} \lesssim_{q,r,\tilde{q}, \tilde{r}} \norm{F}_{L_t^{\tilde{q}^{\prime} } L_x^{\tilde{r}^{\prime}}   (\R \times \R^2)} .
\end{align*}
\end{lem}
For the proofs, see \cite{GG, Y, KT} for details. Also see \cite{T1} for a proof. 

Then we also recall the following bilinear Strichartz estimate: 
\begin{lem}[Bilinear estimates in \cite{B1, B4}]\label{lem Bilinear}
If $\hat{u}_0$ is supported on $\abs{\xi} \sim N$, $\hat{v}_0$ is supported on $\abs{\xi} \sim M$, $M \ll N$,
\begin{equation*}
\norm{(e^{it\Delta} u_0) (e^{it\Delta}v_0)}_{L_t^2 L_x^2 (\R \times \R^2)} \lesssim \parenthese{\frac{M}{N}}^{\frac{1}{2}} \norm{u_0}_{L_x^2 (\R^2)} \norm{v_0}_{L_x^2 (\R^2)} .
\end{equation*}
\end{lem}

\begin{rmk}
Lemma \ref{lem Bilinear} also holds for $(e^{it\Delta} u_0) (\overline{e^{it\Delta}v_0})$. In fact, 
\begin{align*}
\norm{(e^{it \Delta} u_0 ) (e^{it \Delta} v_0)}_{L_t^2 L_x^2}^2 & = \norm{(e^{it \Delta} u_0 )(e^{it \Delta} v_0 ) (\overline{e^{it \Delta} u_0}) (\overline{e^{it \Delta} v_0})}_{L_t^1 L_x^1}   = \norm{(e^{it \Delta} u_0 ) (\overline{e^{it \Delta} v_0})}_{L_t^2 L_x^2}^2  .
\end{align*}
\end{rmk}

\subsection{Profile decomposition}
We first define
\begin{defn}[Symmetry group]
For any position $x_0 \in \R^2$ and scaling parameter $\lambda >0$, we define a unitary transformations $g_{x_{0}, \lambda} : \dot{H}_x^{\frac{1}{2}}(\R^2) \to \dot{H}_x^{\frac{1}{2}} (\R^2)$ by
\begin{equation*}
[g_{x_0 , \lambda} f](x) : =\lambda^{-\frac{1}{2}} f(\lambda^{-1} (x-x_0)).
\end{equation*}
We let $G$ denote the group of such transformations.
\end{defn}

As we mentioned, the following profile decomposition argument is the key ingredient of Theorem \ref{thm Exist minimal}. The proof of Lemma \ref{lem Profile decomp} can be adapted to two dimensions the $\dot{H}^{\frac{1}{2}}$ setting from Theorem 1.6 in \cite{K1}.
\begin{lem}\label{lem Profile decomp}
Let $\{ u_n \}_{n \geq 1}$ be a bounded sequence in $\dot{H}^{\frac{1}{2}} (\R^2)$. After passing to a subsequence if necessary, there exist functions $\{ \phi_i \}_{j \geq 1} \subset \dot{H}^{\frac{1}{2}}(\R^2)$, group elements $g_n^j \in G$ (with parameters $x_n^j$ and $\lambda_n^j$), and times $t_n^j \in \R$ such that for all $J \geq 1$, we have the following decomposition:
\begin{equation*}
u_n = \sum_{j=1}^{J} g_n^j e^{it_n^j \Delta} \phi^j + w_n^j.
\end{equation*}
This decomposition satisfies the following properties:
\begin{enumerate}
\item
For each $j$, either $t_n^j \equiv 0 $ or $t_n^j \to \pm \infty$.
\item
For $J \geq 1$, we have the following asymptotic orthogonality condition:
\begin{equation*}
\frac{\lambda_n^j}{\lambda_n^k} +\frac{\lambda_n^k}{\lambda_n^j} + \frac{\abs{x_n^j -x_n^k}^2}{\lambda_n^j \lambda_n^k} + \frac{\abs{t_n^j (\lambda_n^j)^2 -t_n^k (\lambda_n^k)^2}}{\lambda_n^j \lambda_n^k} \to \infty \text{ as } n \to \infty .
\end{equation*}

\item
For all $n$ and all $J \geq 1$, we have $w_n^J \in \dot{H}^{\frac{1}{2}} (\R^2)$, with 
\begin{equation*}
\lim_{J \to \infty} \limsup_{n \to \infty} \norm{e^{it\Delta} w_n^J}_{L_{t}^8 L_x^8 (\R \times \R^2)} =0 .
\end{equation*}

\end{enumerate}
\end{lem}

\subsection{Local theory and stability} \label{sec LWP and PTB}
In this section we review the local well posedness and stability theory for the Cauchy problem \eqref{NLS}. We adapt to two dimensions the $\dot{H}^{\frac{1}{2}}$- local well-posedness theory and stability in \cite{KM3}. This type of stability result was first shown in \cite{CKSTT2}. In our context, however, we rely on a slight modification of the stability result which only requires space-time bounds on the error itself rather than on its half derivative (see 
Theorem \ref{thm Stability} \eqref{eq Stability} below), for convenience. 
This modification appeared in \cite{Mu2} in dimensions four and above.



\begin{thm}[Standard local well-posedness]\label{thm LWP}
Assume $u_0 \in \dot{H}^{\frac{1}{2}}(\R^2)$, $t_0 \in I$, $\norm{ u_0}_{  \dot{H}^{\frac{1}{2}}(\R^2)  } \leq A$. Then there exists $\delta = \delta(A)$ such that if $\norm{ e^{i(t-t_0)\Delta} u_0}_{L_t^{12} L_t^{6} (I \times \R^2)} < \delta $, there exists a unique solution $u$ to \eqref{NLS}  in $I \times \R^2 $, with $u \in C(I;  \dot{H}^{\frac{1}{2}}(\R^2))$:
\begin{equation*}
\norm{\abs{\nabla}^{\frac{1}{2}} u}_{L_t^3 L_x^6 (I \times \R^2)} + \sup_{t \in I} \norm{\abs{\nabla}^{\frac{1}{2}} u}_{L_x^2} \leq CA, \quad \norm{ u}_{L_t^{12} L_x^6(I \times \R^2)} \leq 2\delta .
\end{equation*}
Moreover, if $u_{0,k} \to u_0$ in $\dot{H}^{\frac{1}{2}}(\R^2)$, obtain that the corresponding solutions $u_k \to u$ in $C(I;  \dot{H}^{\frac{1}{2}}(\R^2))$.
\end{thm}




The proofs of the results above follow from those in \cite{KM1}, once we adapt their norms to our two dimensional context.

\begin{thm}[Stability]\label{thm Stability}
Let $I$ be a compact time interval, with $t_0 \in I$. Suppose $\tilde{u}$ is a solution to 
\begin{equation*}
i\partial_t u+ \Delta u = F(u) +e ,
\end{equation*}
with $\tilde{u}(t_0) =\tilde{u}_0$. 
Suppose 
\begin{equation*}
\norm{\tilde{u}}_{\dot{S}^{\frac{1}{2} } (I)} \leq E \text{ and } 
\norm{\abs{\nabla}^{\frac{1}{2}} e }_{L_t^{\frac{4}{3}} L_x^{\frac{4}{3}}(I \times \R^2)} \leq E
\end{equation*}
for some $E >0$. Let $u_0 \in \dot{H}^{\frac{1}{2}}(\R^2) $ and suppose we have the smallness conditions 
\begin{equation}\label{eq Stability}
\norm{u_0- \tilde{u}_0}_{\dot{H}^{\frac{1}{2}}(\R^2) } + \norm{e}_{L_t^{\frac{12}{5}} L_x^{\frac{6}{5}}(I \times \R^2)} \leq \varepsilon
\end{equation}
for some small $0< \varepsilon < \varepsilon_1 (E)$.

Then, there exists $u : I \times  \R^2 \to \C$ solving \eqref{NLS} with $u(t_0) =u_0$, and there exists $0< c<1$ such that 
\begin{equation*}
\norm{u- \tilde{u}}_{L_t^8 L_x^8 ( I \times \R^2)} \lesssim_E \varepsilon^c .
\end{equation*}
\end{thm}

\section{Impossibility of finite-time blow-up solutions}\label{sec No finite-time blow-up}
In this section we rule out the existence of finite-time blow-up solutions in Theorem \ref{thm 2 types}. As we mentioned before, in $\dot{H}^{\frac{1}{2}}$ critical regime, due to the scaling, the long time Strichartz estimates do not provide any additional decay as we desired in the $s_c = \frac{1}{2}$ setting. So to rule out the existence of finite-time blow-up solutions, we consider the rate of change in time of the mass of solutions restricted within a spacial bump instead. 

\begin{thm}[Impossibility of finite-time blow-up solutions]\label{thm No finite-time blow-up}
If $u$ is an almost periodic solution to \eqref{NLS} in the form of Theorem \ref{thm 2 types} and $T_{max} < \infty$, then $u \equiv 0$.
\end{thm}

\begin{proof}
Consider the following quantity $y^2(t,R)$:
\begin{equation}\label{eq y^2}
y^2 (t,R) : = \int_{\R^2} \chi_R (x) \abs{u (t,x)}^2  \, dx,
\end{equation}
where $\chi_R (x) = \chi (\frac{x}{R})$ is a smooth cutoff function, such that
\begin{equation*}
\chi (x) =
\begin{cases}
1 , &  \text{ if } \abs{x} \leq 1\\
0 , &  \text{ if } \abs{x} >2 ,
\end{cases}
\end{equation*}
and
\begin{equation}\label{eq Chi bdd}
\norm{\abs{\nabla}^k \chi (x)}_{L_x^{\infty} (\R^2)} \leq  1 \quad \text{ for } k = \frac{3}{2} , 3.
\end{equation}
In fact, the quantity defined in \eqref{eq y^2} can be think of the mass of the solution $u$ restricted within the spacial bump with radius $2R$.  

Then we compute the rate of change in time of $y^2(t,R)$, that is, the derivative of $y^2 (t,R)$ with respect to time $t$. 
Using the NLS equation \eqref{NLS} and the properties of complex numbers, we now can pass the time derivative inside the integral and write
\begin{align*}
\frac{\partial y^2}{\partial t}  = 2 \int_{\R^2} \chi_R(x) \re [u_t \bar{u}] \, dx =  - 2\im \int_{\R^2}  \chi_R (x) \Delta u \bar{u}  \, dx .
\end{align*}

\noindent Preforming integration by parts and the product rule in calculus, we have
\begin{equation}
 - \im \int_{\R^2}  \chi_R (x) \Delta u \bar{u}  \, dx  = \im \int_{\R^2} \nabla \parenthese{\chi_R (x) \bar{u}} \cdot \nabla u  \, dx = \im \int_{\R^2} \nabla \chi_R (x) \cdot  \nabla u \bar{u}  \, dx  . \label{eq Imaginary prod}
\end{equation}
Indeed, $\im \int_{\R^2} \nabla \bar{u} \cdot \nabla u \chi_R(x)  \,  dx $ is zero, since $\nabla \bar{u} \cdot \nabla u = \abs{\nabla u}^2$ is a real number.

Now to estimate the time derivative of $y^2(t, R)$ is equivalent to estimate \eqref{eq Imaginary prod}. Notice that there is a product of three functions, hence it is natural to employ Littlewood-Paley decomposition and treat these three functions at different frequency scales separately, i.e. consider 
\begin{equation*}
\sum\limits_{N_1, N_2, N_3} \int_{\R^2}  P_{N_1} (\nabla \chi_R ) \, P_{N_2} ( \nabla u ) \, P_{N_3}  \bar{u}  \, dx .
\end{equation*}
In fact, by Littlewood-Paley decomposition, there are only the following three cases:
\begin{equation*}
\begin{cases}
\text{Case 1: } N_1 \sim N_2 \geq N_3 \\
\text{Case 2: } N_2 \sim N_3 \geq N_1 \\
\text{Case 3: } N_3 \sim N_1 \geq N_2 
\end{cases}
\end{equation*}

\noindent Next, we will deal with these three cases one by one. The remainder of the proof in this section all space norms are over $ \R^2$.

\noindent $\bullet$ {\bf Case 1:} $N_1 \sim N_2 \geq N_3$. 

First, by H\"older's, Bernstein, $N_1 \sim N_2$ and Sobolev embedding, we get
\begin{align}\label{eq Sum}
& \quad \int_{\R^2} P_{N_1} (\nabla \chi_R ) \, P_{N_2} ( \nabla u ) \, P_{N_3}  \bar{u}  \,  dx \notag \\
& \lesssim \norm{P_{N_1} \nabla \chi_R}_{L_x^{\frac{82}{21}}} \norm{P_{N_2} \nabla u}_{L_x^2 } \norm{P_{N_3} u }_{L_x^{\frac{41}{10}} } \notag \\
& \simeq N_1 \norm{P_{N_1} \nabla \chi_R}_{L_x^{2} } \norm{P_{N_2} \abs{\nabla}^{\frac{1}{2}} u}_{L_x^2 } \norm{P_{N_3} \abs{\nabla}^{\frac{1}{2}}u }_{L_x^2 } \parenthese{\frac{N_3}{N_2}}^{\frac{1}{82}} .
\end{align}
Then we are left to sum $N_1$, $N_2$ and $N_3$ over \eqref{eq Sum}. For the sums over $N_2$ and $N_3$, by Cauchy-Schwarz, we arrive at
\begin{equation*}
\sum_{N_3 \leq N_2} \norm{P_{N_2} \abs{\nabla}^{\frac{1}{2}} u}_{L_x^2   } \norm{P_{N_3} \abs{\nabla}^{\frac{1}{2}}u}_{L_x^2   } \parenthese{\frac{N_3}{N_2}}^{\frac{1}{82}}  \sim \norm{ \abs{\nabla}^{\frac{1}{2}} u }_{L_x^2   }^2 .
\end{equation*}
Note that the power of the fraction $\frac{N_3}{N_2}$ here is not necessary to be $\frac{1}{82}$. In fact, any positive power $\varepsilon$ will help us to sum the two summations above, and without the factor $(\frac{N_3}{N_2})^{\varepsilon}$, this sum would be not summable. We choose this number $\frac{1}{82}$ only to avoid the confusion caused by the notation $\varepsilon$.

Now we are left to sum $N_1$ over \eqref{eq Sum}. Splitting the sum over $N_1$ into $\sum_{N_1 < 1} $ and $\sum_{N_1 \geq 1}$, taking the component in \eqref{eq Sum} only depending on $N_1$ and applying Bernstein, Cauchy-Schwarz, H\"older inequalities and compactness of the smooth cutoff function $\chi_R (x)$, we obtain
\begin{align}
\sum_{N_1} N_1  \norm{P_{N_1} \nabla \chi_R }_{L_x^2  } 
& \lesssim \sum_{N_1 < 1 } N_1^{\frac{1}{2}} \norm{P_{N_1} \abs{\nabla}^{\frac{3}{2}} \chi_R }_{L_x^2  }  + \sum_{N_1 \geq 1 } \frac{1}{N_1} \norm{ P_{N_1}  \nabla^3 \chi_R}_{L_x^2 }  \notag\\
& \lesssim \norm{\abs{\nabla}^{\frac{3}{2}}  \chi_R}_{L_x^{\infty} } R+ \norm{\nabla^3 \chi_R}_{L_x^{\infty} } R \lesssim \frac{1}{\sqrt{R}}  \label{eq N_1}\ .\end{align}
The last inequality is true because \eqref{eq Chi bdd} implies $\norm{\abs{\nabla}^k \chi_R (x)}_{L_x^{\infty}} \leq \frac{1}{R^k}$.

\noindent Finally by putting the information above together, we obtain that
\begin{align*}
& \quad \sum_{N_1 \sim  N_2 \geq N_3}\int_{\R^2} P_{N_1} (\nabla \chi_R (x)) P_{N_2} ( \nabla u )P_{N_3}  \bar{u}  \,  dx  \\
&\lesssim \sum_{N_1} \sum_{N_3 \leq N_2} N_1 \norm{P_{N_1} \nabla \chi_R}_{L_x^{2}} \norm{P_{N_2} \abs{\nabla}^{\frac{1}{2}} u}_{L_x^2} \norm{P_{N_3} \abs{\nabla}^{\frac{1}{2}}u }_{L_x^2} \parenthese{\frac{N_3}{N_2}}^{\frac{1}{82}} \lesssim \frac{1}{\sqrt{R}} .
\end{align*}

\noindent $\bullet$ {\bf Case 2:} $ N_2 \sim N_3 \geq N_1 $.

In this case, $N_2 \sim N_3$, we can pass half derivative from $\nabla u$ to $\bar{u}$, then bound these two terms by $\dot{H}^{\frac{1}{2}}$ norm of the solution. By H\"older's and Bernstein, we have
\begin{align*}
 \int_{\R^2} P_{N_1} (\nabla \chi_R ) P_{N_2} ( \nabla u )P_{N_3}  \bar{u}  \,  dx 
& \lesssim N_1 \parenthese{\frac{N_2}{N_3}}^{\frac{1}{2}} \norm{P_{N_1} \nabla \chi_R }_{L_x^{2}} \norm{P_{N_2} \abs{\nabla}^{\frac{1}{2}} u}_{L_x^2} \norm{P_{N_3} \abs{\nabla}^{\frac{1}{2}}u}_{L_x^2}  .
\end{align*}
Again by Cauchy-Schwarz and \eqref{eq N_1}, we have
\begin{align*}
\sum_{N_2 \sim  N_3 \geq N_1}\int_{\R^2} P_{N_1} (\nabla \chi_R ) P_{N_2} ( \nabla u )P_{N_3}  \bar{u}   \, dx   \lesssim \sum_{N_1} N_1 \norm{P_{N_1} \nabla \chi_R }_{L_x^{2}}   \lesssim \frac{1}{\sqrt{R}}.
\end{align*}

\noindent $\bullet$ {\bf Case 3:} $ N_3 \sim N_1 \geq N_2 $.

In this case, $N_2 \leq N_3$, we use the same idea in Case 2, that is, we pass half derivative from $\nabla u$ to $\bar{u}$. By H\"older and Bernstein, we have that
\begin{align*}
\int_{\R^2}  P_{N_1} (\nabla \chi_R ) P_{N_2} ( \nabla u )P_{N_3}  \bar{u}  \,  dx 
& \lesssim N_1 \parenthese{\frac{N_2}{N_3} }^{\frac{1}{2}}  \norm{P_{N_1} \nabla \chi_R }_{L_x^{2}} \norm{P_{N_2} \abs{\nabla}^{\frac{1}{2}} u}_{L_x^2} \norm{P_{N_3} \abs{\nabla}^{\frac{1}{2}}u}_{L_x^2}  .
\end{align*}
Using Cauchy-Schwarz and \eqref{eq N_1} again, 
\begin{align*}
\sum_{N_3 \sim  N_1 \geq N_2}\int_{\R^2}  P_{N_1} (\nabla \chi_R ) P_{N_2} ( \nabla u )P_{N_3}  \bar{u}  \,  dx  \lesssim \frac{1}{\sqrt{R}}   .
\end{align*}

Hence, all these three cases give us the same estimate for $\frac{\partial y^2 }{\partial t}$:
\begin{equation}
\abs{\frac{\partial y^2 }{\partial t}} \lesssim \frac{1}{\sqrt{R}}  .
\end{equation}

At this point, we claim that $y^2 (T_{max}, R) =0$, that is, for any $R$ fixed,
\begin{equation}\label{eq Claim}
\lim_{t \to T_{max}} y^2 (t, R)  = \lim_{t \to T_{max}} \int_{\R^2}  \chi_R(x) \abs{u (t,x)}^2  \, dx=0  .
\end{equation}
Assuming that \eqref{eq Claim} is true, we are able to complete the proof. In fact, by the definition of limit, we fix an arbitrary small number $\varepsilon$, then there exists a $t_0$, such that $y^2 (t_0, R) < \varepsilon$. We can think of $y^2(t,R)$ as a function of $t$, and the slope of $y^2$ is bounded by $\frac{1}{\sqrt{R}}$, then $y^2$ itself should be bounded by
\begin{equation*}
y^2( t, R) \lesssim \frac{ \abs{t_0 -t} }{\sqrt{R}} + \varepsilon  \leq \frac{T_{max} }{\sqrt{R}} + \varepsilon \quad \text{ for any } t<T_{max},
\end{equation*}
especially, 
\begin{equation*}
y^2 (0,R) \lesssim \frac{T_{max} }{\sqrt{R}} + \varepsilon .
\end{equation*}
Then we let $R$ go to infinity, it is easy to see that 
\begin{equation*}
\lim_{R \to \infty }y^2(0,R) \lesssim \varepsilon \quad \text{ for any arbitrary } \varepsilon,
\end{equation*}
which implies
\begin{equation*}
\lim_{R \to \infty }y^2(0,R) =0 .
\end{equation*}
Therefore, by passing the limit inside, we have
\begin{equation*}
0= \lim_{R \to \infty }y^2(0,R) = \lim_{R \to \infty } \int_{\R^2} \chi_R (x) \abs{u(0,x)}^2  \, dx  = \int_{\R^2}  \abs{u(0,x)}^2 \,  dx = \norm{u_0}_{L_x^2}^2 .
\end{equation*}
It implies $u_0 \equiv 0 $, which contradicts with the fact that $u$ is a minimal blow-up solution, so we prove the impossibility of finite-time blow-up solutions.

However, it still remains to prove our claim \eqref{eq Claim} above:
\begin{proof}[Proof of \eqref{eq Claim}]
By the definition of limit, it is equivalent to prove that: for any $\varepsilon > 0$, there exists $t_0$, such that for any $t > t_0$, 
\begin{equation*}
\int_{\R^2} \chi_R(x) \abs{u (t,x)}^2 dx < \varepsilon  .
\end{equation*}

By the almost periodicity of the solution (Definition \ref{defn AP}), we know that for $\eta = \frac{\varepsilon}{2R}$ fixed, there exist $c(\eta)$, $N(t)$, and $x(t)$ such that 
\begin{equation*}
\parenthese{\int_{\abs{x-x(t)} > \frac{c(\eta)}{N(t)}} \abs{u(t,x )}^4  \, dx}^{\frac{1}{2}} < \eta = \frac{\varepsilon}{2R}.
\end{equation*}

Hence, if we consider the cutoff mass $y^2(t,R)$ inside the bump $\mathbf{1}_{\abs{x} \leq \frac{c(\eta)}{N(t)}}$ and outside the $\chi_R \mathbf{1}_{\abs{x} > \frac{c(\eta)}{N(t)}}$ separately, the mass inside will be small because the measure of the bump is small and the mass outside will be also small due to almost periodicity of the solution. By H\"older, the almost periodicity of the solution, Sobolev embedding and uniform $\dot{H}^{\frac{1}{2}}$ norm bound for the solution, we have
\begin{align*}
y^2(t,R) &
\leq \int \mathbf{1}_{\abs{x-x(t)} \leq \frac{c(\eta)}{N(t)}} \abs{u}^2  \, dx + \int \chi_R \mathbf{1}_{\abs{x-x(t)} > \frac{c(\eta)}{N(t)}} \abs{u}^2  \, dx \\
& \lesssim \norm{\mathbf{1}_{\abs{x-x(t)} \leq \frac{c(\eta)}{N(t)}} }_{L_x^2} \norm{u}_{L_x^4}^2 + \norm{\chi_R }_{L_x^2} \norm{u \mathbf{1}_{\abs{x-x(t)} > \frac{c(\eta)}{N(t)}}}_{L_x^4}^2 \\
& \lesssim  \frac{c(\eta) }{N(t)} \norm{u}_{L_t^{\infty} \dot{H}_x^{\frac{1}{2}}}^2 + \frac{\varepsilon}{2}   .
\end{align*}
The first term $\frac{c(\eta)}{N(t)} \norm{u}_{L_t^{\infty} \dot{H}_x^{\frac{1}{2}}([0,T_{max} ) \times \R^2)}^2 $ can be made less than $\frac{\varepsilon}{2}$. In fact $N(t)$ goes to infinity as $t $ approaches to $T_{max}$ by Corollary \ref{cor N(t) blow-up}, so it is always possible for us to choose some $t$ close enough to $T_{max}$ such that $N(t)$ is large enough. Then the claim follows.
\end{proof}

The proof of Theorem \ref{thm No finite-time blow-up} is complete.
\end{proof}

\section{Atomic spaces and $X$-norm}\label{sec Atomic spaces}

In this section, we recall some basic definitions and properties of atomic spaces, then prove a decomposition lemma, which will be used in the proof of the long time Strichartz estimates in Section \ref{sec LTS}. At the end of this section, we give the definition of $\tilde{X}_{k_0}$ norm, which will be used in defining the long time Strichartz estimates in Section \ref{sec LTS}.

\subsection{Basic definitions and properties of the atomic space}
The atomic spaces were first introduced in partial differential equations in \cite{KT1}, and then applied to KP-II equations in \cite{HHK} and nonlinear Schr\"odinger equations in \cite{KT2, KT3, HTT1, HTT2}. They are useful in many different settings, and it is worth mentioning that they are quite helpful in fixing the lack of Strichartz endpoint problems.

First, we recall some basic definitions and properties of the atomic space. Let $\mathcal{Z}$ denote the set of finite partitions 
\begin{equation*}
-\infty < t_0 < t_1 < \dots < t_K \leq \infty
\end{equation*}
of the real line. If $t_K =\infty$, we use the convention that $v(t_K) : = 0$ for all functions $v: \R \to L^2$. Let $\chi_I : \R \to \R$ denote the sharp characteristic function of a set $I \subset \R$.

\begin{defn}[$U^p$ spaces, Definition 2.1 in \cite{HHK}]
Let $1\leq p < \infty$. For $\bracket{t_k}_{k=0}^K \in \mathcal{Z}$ and $\bracket{\phi_k}_{k=0}^{K-1} \subset L^2$ with $\sum\limits_{k=0}^{K-1} \norm{\phi_k}_{L^2}^p = 1$, we call
the piecewise defined function $a : \R \to L^2$:
\begin{equation*}
a = \sum_{k=1}^K \chi_{[t_{k-1}, t_k)} \phi_{k-1}
\end{equation*}
a $U^p$-atom, and we define the atomic space $U^p(\R,L^2)$ of all functions $u : \R \to L^2$ such that
\begin{equation*}
u = \sum_{j=1}^{\infty} \lambda_{j} a_j \quad \text{for } U^p \text{-atoms }  a_j, \bracket{\lambda_j} \in l^1,
\end{equation*}
with norm
\begin{equation*}
\norm{u}_{U^p} : = \inf \bracket{\sum_{j=1}^{\infty} \abs{\lambda_j} : u = \sum_{j=1}^\infty \lambda_j a_j,\lambda_j \in \C \text{ and } a_j \text{ are }  U^p \text{-atoms}}.
\end{equation*}

\noindent If $J \subset \R$ is an interval, then we say that $u_{\lambda}$ is a $U^p(J)$-atom if $t_k \in J$ for all $1 \leq k \leq K$. Then for any $1 \leq p < \infty$, let
\begin{equation*}
\norm{u}_{U^p (J \times \R^2)} = \inf \bracket{\sum_{j=1}^{\infty} \abs{\lambda_j} : u= \sum_{j=1}^{\infty} \lambda_j a_j , \lambda \in \C, a_j \text{ are }  U^p(J) \text{-atoms}} .
\end{equation*}
For any $1 \leq p < \infty$, $U^p (J \times \R^2) \hookrightarrow L^{\infty} L^2 (J \times \R^2)$. Additionally, $U^p$-functions are continuous except at countably many points and right-continuous everywhere.
\end{defn}

\begin{defn}[$V^p$ spaces, Definition 2.3 in \cite{HHK}]
Let $1 \leq p < \infty$. 
\begin{enumerate}
\item
We define $V^p (\R, L^2)$ as the space of all functions $v : \R \to L^2$ such that
\begin{equation}\label{eq V^p}
\norm{v}_{V^p} : = \sup_{\bracket{t_k}_{k=0}^K\in \mathcal{Z}} \parenthese{\sum_{k=1}^K \norm{v(t_k)-v(t_{k-1})}_{L^2}^p}^{\frac{1}{p}} 
\end{equation}
is finite. Notice that here we use the convention $v (\infty) =0$.

\item
Likewise, let $V_{rc}^p (\R ,L^2)$ denote the closed subspace of all right-continuous functions: $v : \R \to L^2$ such that $\lim\limits_{t \to -\infty} v(t) =0$, endowed with the same norm \eqref{eq V^p}.

\item
If $J \subset \R$, then
\begin{equation*}
\norm{v}_{V^p (J \times \R^2)} : = \sup_{\bracket{t_k}_{k=0}^K\in \mathcal{Z}} \parenthese{\norm{v(t_0)}_{L^2}^p + \sum_{k=1}^K \norm{v(t_k)-v(t_{k-1})}_{L^2}^p + \norm{v(t_K)}_{L^2}^p}^{\frac{1}{p}} ,
\end{equation*}
where each $t_k$ lies in $J$. Note that $\{ t_k \}$ may be a finite or infinite sequence.
\end{enumerate}

\end{defn}

\begin{prop}[Embedding, Proposition 2.2, Proposition 2.4 and Corollary 2.6 in \cite{HHK}]
For $1\leq p \leq q < \infty$,
\begin{equation*}
U^p(\R, L^2) \hookrightarrow U^q(\R, L^2) \hookrightarrow  L^{\infty}(\R,L^2),
\end{equation*}
and functions in $U^p(\R, L^2)$ are right continuous and $\lim\limits_{t\to -\infty} u(t) = 0$ for each $u \in U^p(\R , L^2)$.

The Banach subspace of all right continuous functions endowed with $\norm{\, \cdot \,}_{V^p}$ is denoted by $V_{rc}^p (\R, L^2)$. Note that,
\begin{equation*}
U^p(\R, L^2) \hookrightarrow V_{rc}^p (\R, L^2) \hookrightarrow L^{\infty}(\R,L^2).
\end{equation*}
Moreover, for $1\leq p \leq q < \infty$,
\begin{equation*}
U^p(\R, L^2) \hookrightarrow V_{rc}^p (\R, L^2) \hookrightarrow U^q(\R, L^2) \hookrightarrow  L^{\infty}(\R,L^2) .
\end{equation*}
\end{prop}

\begin{defn}[$U_{\Delta}^p$ and $V_{\Delta}^p$ spaces, Definition 2.15 in \cite{HHK}]
For $s \in \R$, we let $U_{\Delta}^p L_x^2$ (respectively $V_{\Delta}^p L_x^2$) be the space of all functions $u : \R \to L^2_x(\R^d)$ such that $t \mapsto e^{-it\Delta}u(t)$ is in $U^p(\R, L_x^2)$ (respectively in  $V^p(\R, L_x^2)$), with norms
\begin{equation*}
\norm{u}_{U_{\Delta}^p L^2} := \norm{e^{-it\Delta}u(t)}_{U^p(\R, L_x^2)}, \quad \norm{u}_{V_{\Delta}^p L^2} := \norm{e^{-it\Delta}u(t)}_{V^p(\R,L_x^2)}.
\end{equation*}
\end{defn}

\begin{prop}\label{prop Embed}
For $1\leq p \leq q < \infty$,
\begin{equation*}
U_{\Delta}^p(\R, L^2) \hookrightarrow V_{\Delta}^p (\R, L^2) \hookrightarrow U_{\Delta}^q(\R, L^2) \hookrightarrow  L^{\infty}(\R,L^2) .
\end{equation*}
\end{prop}

\begin{lem}[(29) in \cite{KT3}, Lemma 3.3 in \cite{D3}]\label{lem U^p sum}
Suppose $J= I_1 \cup I_2$, $I_1 =[a,b]$, $I_2 = [b,c]$, $a \leq b \leq c$. Then
\begin{equation*}
\norm{u}_{U_{\Delta}^p (J \times \R^2)}^p \leq \norm{u}_{U_{\Delta}^p (I_1 \times \R^2)}^p + \norm{u}_{U_{\Delta}^p (I_2 \times \R^2)}^p .
\end{equation*}

\end{lem}

\begin{prop}[Duality, Theorem 2.8 in \cite{HHK}]\label{prop Duality}
Let $DU_\Delta^p$ be the space of functions
\begin{equation*}
DU_\Delta^p = \bracket{(i\partial_t + \Delta) u : u \in U_\Delta^p },
\end{equation*}
and the $DU_\Delta^p = (V_\Delta^{p^{\prime}})^{\ast}$, with $\frac{1}{p} + \frac{1}{p^{\prime}} = 1$. Then $(0 \in J )$
\begin{equation*}
\norm{\int_0^t e^{i(t-t^{\prime})\Delta}F(u)(t^{\prime})\, dt^{\prime}}_{U_\Delta^p (J\times \R^d)} \lesssim \sup\bracket{\int_J \inner{ v, F}  \,  dt: \norm{v}_{V_\Delta^{p^{\prime}} }=1 }.
\end{equation*}
\end{prop}

\begin{lem}[Decomposition lemma]\label{lem Decomp}
Suppose $[a,b]=J = \cup_{k=1}^K P_k$, where $P_k=[a_k, b_k]$ $(b_k=a_{k+1})$ are consecutive intervals. Also suppose that $\abs{\nabla}^{\frac{1}{2}}  F \in L_t^1 L_x^2 (J \times \R^2)$, then for any $t_0 \in J$,
\begin{equation*}
\norm{\abs{\nabla}^{\frac{1}{2}}   \int_{t_0}^t  e^{i(t-t^{\prime})\Delta} F(t^{\prime})  \, dt^{\prime}  }_{U_{\Delta}^2 (J\times \R^2)} \lesssim \sum_{k=1}^K \sup_{\norm{v}_{V_{\Delta}^2 (P_k \times \R^2) } =1} \int_{P_k} \inner{v, \abs{\nabla}^{\frac{1}{2}}  F}  \, dt  .
\end{equation*}
Note that the implicit constant will not depend on $\norm{\abs{\nabla}^{\frac{1}{2}}  F}_{L_t^1 L_x^2}$.
\end{lem}

\begin{proof}
We first consider $t> t_0$, and $t_0 \in P_{k^*}=[a_{k^*}, b_{k^*}]$. Then by Proposition \ref{prop Duality} and the partition of the interval $J$,  we write
\begin{align*}
& \quad \norm{\abs{\nabla}^{\frac{1}{2}}   \int_{t_0}^t  e^{i(t-t^{\prime})\Delta} F(t^{\prime})  \,  dt^{\prime}  }_{U_{\Delta}^2 ([t_0, b]\times \R^2)}  \lesssim \sup_{\norm{v}_{V_{\Delta}^2([t_0, b] \times \R^2) =1}} \int_{[t_0, b]} \inner{v, \abs{\nabla}^{\frac{1}{2}} F} \, dt \\
& \leq \sum_{k>k^*} \sup_{\norm{v}_{V_{\Delta}^2(P_k \times \R^2) }=1} \int_{ P_k} \inner{ v  , \abs{\nabla}^{\frac{1}{2}} F} \, dt + \sup_{\norm{v}_{V_{\Delta}^2([a_{k^*},b_{k^*}] \times \R^2)} =1}  \int_{ [a_{k^*},b_{k^*}] } \inner{  v , \abs{\nabla}^{\frac{1}{2}} F} \, dt \\
& = \sum_{k \geq k^*} \sup_{\norm{v}_{V_{\Delta}^2 (P_k \times \R^2) } =1} \int_{P_k} \inner{v, \abs{\nabla}^{\frac{1}{2}}  F}  \, dt  .
\end{align*}

Similar argument holds for $t < t_0$. 
Therefore, the lemma follows.
\end{proof}

\begin{prop}[Transfer Principle, Proposition 2.19 in \cite{HHK}]\label{prop TransferPrinciple}
Let 
\begin{equation*}
T_0 : L^2 \times \dots \times L^2 \to L_{loc}^{1}
\end{equation*}
be an m-linear operator. Assume that for some $1\leq p,\ q \leq \infty$
\begin{equation*}
\norm{ T_0 (e^{it\Delta} \phi_1,\dots, e^{it\Delta} \phi_m)}_{L^p(\R, L_x^q)} \lesssim \prod_{i=1}^m\norm{\phi_i}_{L^2}.
\end{equation*}
Then, there exists an extension $T : U_{\Delta}^p \times \dots \times U_{\Delta}^p \to L^p(\R, L_x^q)$ satisfying
\begin{equation*}
\norm{T(u_1, \dots, u_m)}_{L^p(\R, L_x^q)} \lesssim \prod_{i=1}^m \norm{u_i}_{U_{\Delta}^p},
\end{equation*}
and  such that $T(u_1, \dots , u_m) (t, \cdot) = T_0 (u_1(t), \dots , u_m(t))(\cdot)$, a.e.
\end{prop}

\begin{cor}\label{cor U Bilinear}
Suppose that under the same condition on the supports of $\hat{u}_0$ and $\hat{v}_0$ as in Lemma \ref{lem Bilinear}, i.e. $\hat{u}_0$ is supported on $\abs{\xi} \sim N$, $\hat{v}_0$ is supported on $\abs{\xi} \sim M$, $M \ll N$, for $\frac{1}{q}+ \frac{1}{q} =1$, $2 \leq p \leq \infty$
\begin{equation*}
\norm{(e^{it\Delta} u_0) (e^{it\Delta}v_0)}_{L_t^p L_x^q (\R \times \R^2)} \lesssim \parenthese{\frac{M}{N}}^{\frac{1}{p}} \norm{u_0}_{L_x^2 (\R^2)} \norm{v_0}_{L_x^2 (\R^2)} .
\end{equation*}
Then if $\hat{u}(t, \xi)$ and $\hat{v}(t,\xi)$ are under the same conditions,
\begin{equation*}
\norm{uv}_{L_t^p L_x^q (I \times \R^2)} \lesssim \parenthese{\frac{M}{N}}^{\frac{1}{p}} \norm{u}_{U_{\Delta}^p (I \times \R^2)}\norm{v}_{U_{\Delta}^p (I \times \R^2)} .
\end{equation*}
\end{cor}

Now we are ready to define the suitable atomic space where we are able to drive a long time Strichartz estimate.

\subsection{$\tilde{X}_{k_0}$-norm}
In this section, we are ready to define the long time Strichartz norm $\tilde{X}_{k_0}$-norm. We first fix some parameters and define some suitable small intervals, then give the construction of this $\tilde{X}_{k_0}$-norm.

\subsubsection{Choice of $\varepsilon_1$, $\varepsilon_2$ and $\varepsilon_3$}
Fix three constants $\varepsilon_1$, $\varepsilon_2$ and $\varepsilon_3$ satisfying
\begin{equation}\label{eq Epsilon123}
0 < \varepsilon_3 \ll \varepsilon_2 \ll \varepsilon_1 < 1 .
\end{equation}
We will add restrictions to these constants later. 

Fix a non-negative integer $k_0$ and suppose that $M = 2^{k_0}$. Let $[a,b]$ be an interval such that
\begin{equation}\label{eq Scale1}
\int_a^b \int \abs{u(t,x)}^8 \, dxdt = M,
\end{equation}
and
\begin{equation}\label{eq Scale2}
\int_a^b N(t) \, dt = \varepsilon_3 M.
\end{equation}
Note that we are always able to choose such interval $[a,b]$, since the scalings of \eqref{eq Scale1} and \eqref{eq Scale2} are different. We will choose $\varepsilon_1$ and $\varepsilon_2$ later in \eqref{eq Epsilon}.

\subsubsection{Definitions of small intervals: $J_l$ and $J^{\alpha}$}
Next, we consider the following two types of partitions of $[a,b]$ ($J_l$ small intervals and $J^{\alpha}$ small intervals):
\begin{defn}[$J_l$ small intervals]\label{defn Small intervals 1}
Let $[a,b] = \cup_{l=0}^{M-1} J_l$, $l = 0, \dots , M-1$ with 
\begin{equation*}
\norm{u}_{L_{t,x}^8 (J_l \times \R^2)}^8 =1 .
\end{equation*}
\end{defn}

\begin{defn}[$J^{\alpha}$ small intervals]\label{defn Small intervals 2}
Let $[a,b] = \cup_{\alpha=0}^{M-1} J^{\alpha}$, $\alpha = 0,\dots, M-1$ such that
\begin{equation*}
\int_{J^{\alpha}} (N(t) + \varepsilon_3 \norm{u(t)}_{L_x^8 (\R^2)}^8) \, dt = 2 \varepsilon_3  .
\end{equation*}

\end{defn}


\begin{rmk}[Differences between $J_l$ and $J^{\alpha}$ small intervals]
The total number of these two types of intervals are the same, since we have fixed $M$ at first. In fact,
\begin{equation*}
\norm{u}_{L_{t,x}^8 (J_l \cup J_{l+1} \times \R^2)}^8 =2 .
\end{equation*}
So every $J^{\alpha}$ small interval will be covered within at most three $J_l$ small intervals. A $J_l$ interval may intersect $J_{\alpha}$ intervals and a $J_{\alpha}$ interval may intersect $J_l$ intervals. The partitions $J_{\alpha}$ and $J_l$ can be quite different.
\end{rmk}

\begin{rmk}\label{rmk Strichartz pair} 
For any Strichartz pair $(q,r)$, 
\begin{equation*}
\norm{ \abs{\nabla}^{\frac{1}{2}}  u  }_{L_{t}^{q} L_{x}^{r}(J_l \times \R^2)} \lesssim_q 1	\text{ and }  \norm{ \abs{\nabla}^{\frac{1}{2}}  u  }_{L_{t}^{q} L_{x}^{r}(J^{\alpha} \times \R^2)} \lesssim_q 1  .
\end{equation*}
\end{rmk}
\begin{proof}[Proof of Remark \ref{rmk Strichartz pair}] 
Subdividing intervals $J_l$ or $J^{\alpha}$ into subintervals such that the $L_{t,x}^8$ norm is sufficiently small and running a standard continuity argument on each subinterval would give the boundness of the norms in this remark.

\end{proof}

\subsubsection{Construction of $G_k^j$ intervals and restrictions of $\varepsilon_1$, $\varepsilon_2$ and $\varepsilon_3$}
\begin{defn}
For an integer $0 \leq j < k_0$, $0 \leq k < 2^{k_0 -j}$, let
\begin{equation*}
G_k^j =\cup_{\alpha =k 2^j}^{(k+1)2^j-1} J^{\alpha}  .
\end{equation*}
For $j \geq k_0$ let $G_k^j =[a,b]$.
\end{defn}

\begin{rmk}\label{rmk N(J_l)}
After defining these two types of small intervals, we recall Remark \ref{rmk N(J)}, then we have
\begin{align*}
\frac{1}{N(J_l)} & = N(J_l ) \cdot N^{-2}(J_l ) \sim N(J_l) \abs{J_l} , \\
\sum_{J_l \subset G_k^j } \frac{1}{N(J_l)} & \sim \sum_{J_l \subset G_k^j }   N(J_l) \abs{J_l} \sim \int_{G_k^j } N(t ) \, dt \sim 2^j \varepsilon_3  .
\end{align*}
\end{rmk}

\noindent Recall \eqref{eq Defn AP2}, Lemma \ref{lem Local const N(t)} and Remark \ref{rmk N(J)}. It is possible to choose $\varepsilon_1$, $\varepsilon_2$ and $\varepsilon_3$ which satisfy \eqref{eq Epsilon123} and also the following conditions, which will be used in Section \ref{sec LTS}:
\begin{equation}\label{eq Epsilon}
\begin{cases}
\begin{aligned}
& \abs{\frac{d}{dt} N(t)} \leq  \frac{N^3(t)}{\varepsilon_1^{1/2}}   \qquad \implies \abs{\frac{d}{dt}\parenthese{\frac{1}{N(t)}}} \leq \frac{N(t)}{\varepsilon_1^{1/2}}  \\
&\int_{\abs{x-x(t)}  \leq \frac{\varepsilon_3^{1/4}}{N(t)}} \abs{ \abs{\nabla}^{\frac{1}{2}} u(t,x)}^2 \, dx + \int_{\abs{\xi} \leq  \varepsilon_3^{1/4}N(t)} \abs{\xi }\abs{\hat{u}(t, \xi )}^2 \, d\xi \leq \varepsilon_2^2 \\
& \varepsilon_3 < \varepsilon_2^{24}
\end{aligned}
\end{cases}
\end{equation}

\begin{rmk}\label{rmk Difference 1/N(t)}
Combine Definition \ref{defn Small intervals 2} and \eqref{eq Epsilon}, then the difference of $\frac{1}{N(t)}$ on $G_{\alpha}^i$ is at most,
\begin{equation*}
\int_{G_{\alpha}^i}\abs{\frac{d}{dt } \parenthese{\frac{1}{N(t)}}}  dt \leq \int_{G_{\alpha}^i} \frac{N(t)}{\varepsilon_1^{1/2}} dt \leq  \varepsilon_1^{-1/2} \varepsilon_3 2^{i}  .
\end{equation*}
\end{rmk}

\subsubsection{Definition of $\tilde{X}_{k_0}$ spaces and properties}
\noindent In the rest of the paper, the Littlewood-Paley projection $P_{2^{-i-2} \leq \cdot \leq 2^{-i+2}}  $ will be abbreviated to $P_{2^{-i}} $.

\begin{defn}[$\tilde{X}_{k_0}$ spaces]\label{defn X}
For any $G_k^j \subset [a,b]$, let
\begin{equation}\label{X}
\begin{aligned}
\norm{u}_{X(G_k^j \times \R^2)}^2   := & \sum_{i: 0 \leq i < j }  2^{i-j} \sum_{G_{\alpha}^i \subset G_k^j} \norm{\abs{\nabla}^{\frac{1}{2}}  P_{ 2^{-i}} u}_{U_{\Delta}^2 (G_{\alpha}^i \times \R^2)}^2 + \sum_{i: i \geq j}  \norm{\abs{\nabla}^{\frac{1}{2}} P_{2^{-i}} u}_{U_{\Delta}^2 (G_{k}^j \times \R^2)}^2 .
\end{aligned}
\end{equation} 
Then define $\tilde{X}_{k_0}$ to be the supremum of \eqref{X} over all intervals $G_k^j \subset [a,b]$ with $k \leq k_0$.
\begin{equation*}
\norm{u}_{\tilde{X}_{k_0} ([a,b] \times \R^2)}^2 : = \sup_{j: 0 \leq j \leq k_0} \sup_{G_k^j \subset [a,b]} \norm{u}_{X(G_k^j \times \R^2)}^2  .
\end{equation*}
Also for $0 \leq k_{\ast}  \leq k_0$, let
\begin{equation*}
\norm{u}_{\tilde{X}_{k_{\ast}} ([a,b] \times \R^2)}^2 : = \sup_{j: 0 \leq j \leq k_{\ast}} \sup_{G_k^j \subset [a,b]} \norm{u}_{X(G_k^j \times \R^2)}^2  .
\end{equation*}
$\norm{u}_{\tilde{X}_{k_{\ast}} (G_k^j \times \R^2)}$, $k_{\ast} \leq j$ is defined in a similar manner:
\begin{equation*}
\norm{u}_{\tilde{X}_{k_{\ast}} (G_k^j \times \R^2)}^2 : = \sup_{i: 0 \leq i \leq k_{\ast}} \sup_{G_{\alpha}^i \subset G_k^j} \norm{u}_{X(G_{\alpha}^i \times \R^2)}^2  .
\end{equation*}
\end{defn}
\noindent Recall that we have no Galilean transformation in our case and frequency center $\xi (t)$ is the origin.

\begin{figure}[!htp]
	\centering
	\includegraphics[height=13cm,width=13cm]{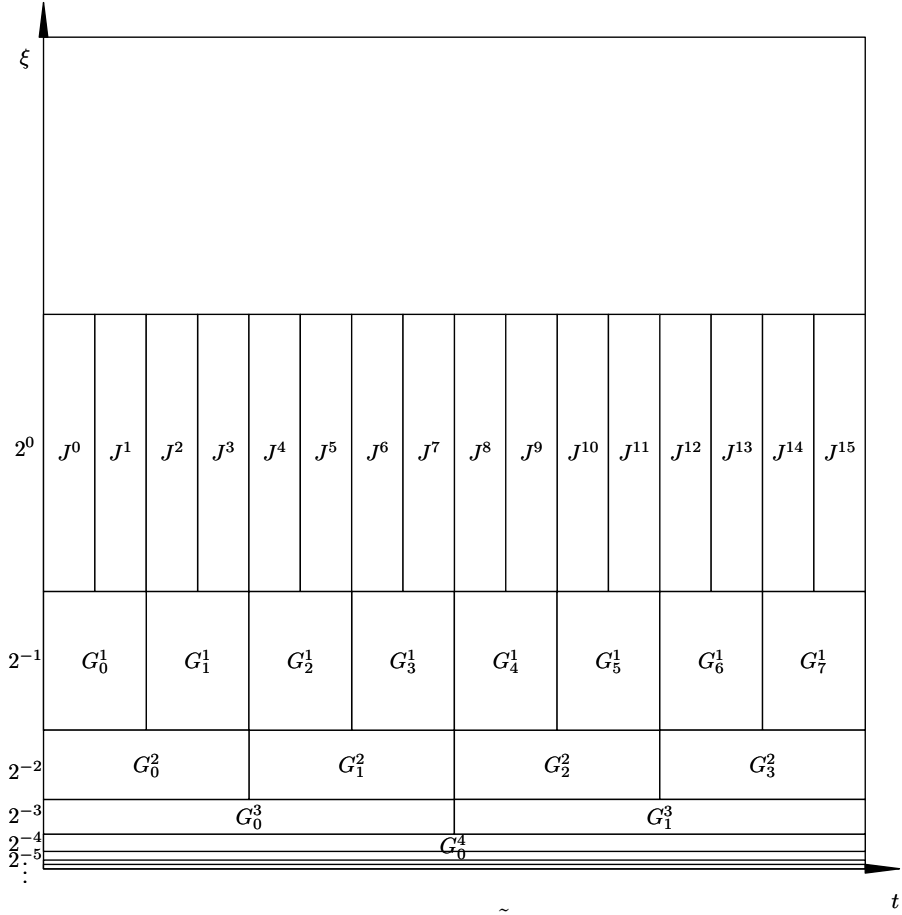}
	\vspace{-1cm}
	\caption{$\tilde{X}_{k_0}$ norm}\label{figure}
\end{figure}

\begin{rmk}[Construction of $G_k^j$ intervals]
For example, in Figure \ref{figure}, we take $k_0 =4$, $M=2^{k_0}= 16$, hence there are $16$ small intervals: $J^0, J^1, \dots , J^{15}$ at frequency level $2^0$. We may treat them as the building blocks in the process of constructing $X$-norm. Then in order to build a lower level $2^{-1}$, we combine every two consecutive small intervals into a larger interval, that is, at a lower level $2^{-1}$, the unions of two consecutive small intervals give us $G_0^1 = J^0 \cup J^1, \dots , G_7^1= J^{14} \cup J^{15}$. Note that upper indices in $G_{k}^j$ indicate the length of the time interval (more precisely, $2^j$ is the number of small intervals inside $G_k^j$) and the lower indices  in $G_{k}^j$ are the locations of the time intervals. In particular, for an interval with its upper indices $0$, it means that it is a  $J^{\alpha}$-type small interval, i.e. $G_k^0 = J^k$. Then we continue moving onto the next level $2^{-2} $ to get $G_0^2,G_1^2,G_2^2$ and $G_3^2$, and even lower levels to get $G_0^3$, $G_1^3$ and $G_0^4$. We can see that $G_0^4$ is the whole interval $[a,b]$, so we stop building intervals here. For any lever lower than $2^{-4} = 2^{k_0}$, we just take the whole interval $[a,b]$.
\end{rmk}

\begin{rmk}[Constriction of the $X$ norm]
We can see the structure of $X$ norm from the figure above. First, we localize the solution $u$ at different frequencies. Then the first term in $X$ norm is:
\begin{equation*}
\sum_{i: 0 \leq i < j }  2^{i-j} \sum_{G_{\alpha}^i \subset G_k^j} \norm{\abs{\nabla}^{\frac{1}{2}}  P_{ 2^{-i}} u}_{U_{\Delta}^2 (G_{\alpha}^i \times \R^2)}^2 .
\end{equation*}
In fact, for a fixed frequency level $2^{-i}$ (higher than $2^{-j}$), compute the average of frequency localized $U_{\Delta}^2$ norms on all the corresponding time intervals $G_{\alpha}^i$'s, then sum over all the frequencies higher than $2^{-j}$. 

\noindent The second term 
\begin{equation*}
\sum_{i: i\geq j}  \norm{\abs{\nabla}^{\frac{1}{2}} P_{2^{-i}} u}_{U_{\Delta}^2 (G_{k}^j \times \R^2)}^2 .
\end{equation*}
is the summation of  frequency localized $U_{\Delta}^2$ norms over all the frequencies lower than $2^{-j}$ on the time interval $G_k^j$.
\end{rmk}

\begin{rmk}[$\tilde{X}_{k_0}$ norm]
After we compute every $X$ norm over interval $G_k^j$, we are only two supremums away from the $\tilde{X}_{k_0}$ norm on the interval $[a,b]$. First, we fix a frequency level $2^{-j}$ and take the supremum over all the intervals at this level. This step picks out the largest candidates from each level horizontally. Then we choose the largest one from these candidates (vertically). This is the second supremum.
\end{rmk}

\begin{prop}[Some properties of $\tilde{X}_{j}(G_k^j \times \R^2)$ norm]\label{prop X properties}
We will use the following estimates in Section \ref{sec No quasi-soliton}. For $i \leq  j$, let $(q,r)$ be any admissible pair, then we have:
\begin{enumerate}
\item
$\norm{  \abs{\nabla}^{\frac{1}{2}} P_{2^{-i}} u }_{ L_{t}^{q} L_{x}^{r} (G_k^j \times \R^2)}   \lesssim_{q,r} 2^{\frac{j-i}{q}}  \norm{u}_{\tilde{X}_j (G_k^j \times \R^2)}  $,
\item
$ \norm{   P_{ > 2^{-i}} u }_{ L_{t}^{q} L_{x}^{r} (G_k^j \times \R^2)}   \lesssim 2^{\frac{j}{2}}\norm{u}_{\tilde{X}_j (G_k^j \times \R^2)} $,
\item
$ \norm{  \abs{\nabla}^{\frac{1}{2}} P_{ \leq 2^{-j}} u }_{ L_{t}^{q} L_{x}^{r} (G_k^j \times \R^2)} \lesssim \norm{u}_{\tilde{X}_j (G_k^j \times \R^2)} $.
\end{enumerate}
\end{prop}

\begin{proof}
Similar calculations as in (3.48) and (3.49) in \cite{D2}, yield Proposition \ref{prop X properties}.
\end{proof}

\section{Long time Strichartz estimate}\label{sec LTS}

In this section, we recall the long time Strichartz estimate introduced in \cite{D3} and prove a long time Strichartz estimate adapted in our $\dot{H}^{\frac{1}{2}}$ setting based on the $X_{k_0}$ norm defined in Definition \ref{defn X}. This long time Strichartz estimate, giving us a good control of low frequency component of the solutions, will be used to in the proof of frequency-localized Morawetz estimates in Section \ref{sec No quasi-soliton}.

\subsection{Long time Strichartz estimate in the mass-critical regime in two dimensions in \cite{D3}}
In dimensions two, the endpoint of Strichartz estimates is false, more precisely: Let $P$ be a Fourier multiplier with symbol in $C_0^{\infty}(\R^2)$ (thus $\widehat{Pf} = \phi \hat{f}$ for some $\phi \in C_0^{\infty} (\R^2)$) which is not identically zero. Then there does not exist a constant $C > 0$ for which one has the estimate
\begin{equation*}
\norm{e^{it\Delta} Pf}_{L_t^2 L_x^{\infty}(\R \times \R^2)} \leq C \norm{f}_{L_x^2 (\R^2)}
\end{equation*}
for all $f \in L_x^2(\R^2)$. This makes us unable to choose the regular Strichartz space. 
However the long time Strichartz estimate highly relies on the double endpoint Strichartz. Therefore, we have to prove new long time Strichartz estimate adapted to two dimensions.

In two dimensional mass-critical regime, Dodson \cite{D3} defined a new space on which to compute the long time Strichartz estimates:
For any $G_k^j \subset [a,b]$, 
\begin{equation}\label{XX}
\begin{aligned}
\norm{u}_{X(G_k^j \times \R^2)}^2  := & \sum_{i : 0 \leq i < j}  2^{i-j} \sum_{G_{\alpha}^i \subset G_k^j} \norm{ P_{ \xi (G_{\alpha}^i), 2^i} u}_{U_{\Delta}^2 (G_{\alpha}^i \times \R^2)}^2  + \sum_{i : i\geq j }  \norm{ P_{\xi(G_k^j), 2^i} u}_{U_{\Delta}^2 (G_{k}^j \times \R^2)}^2.
\end{aligned}
\end{equation}
Then define 
\begin{equation*}
\norm{u}_{\tilde{X}_{k_0} ([a,b] \times \R^2)}^2 : = \sup_{j:0 \leq j \leq k_0} \sup_{G_k^j \subset [a,b]} \norm{u}_{X(G_k^j \times \R^2)}^2 .
\end{equation*}
Dodson showed that $\norm{u}_{\tilde{X}_{k_0} ([0,T] \times \R^2)} \lesssim 1$ as the new long time Strichartz estimate in dimensions two, which played a similar role as the long time Strichartz estimate in dimensions three and higher:
\begin{equation*}
\norm{P_{\abs{\xi -\xi(t)} >N} u}_{L_t^2 L_x^{\frac{2d}{d-2}} (J \times \R^d)} \lesssim \parenthese{\frac{K}{N}}^{\frac{1}{2}} +1 ,
\end{equation*}
where $J$ is an interval satisfying 
\begin{equation*}
\int_J N(t)^3 \, dt = K .
\end{equation*}

\subsection{Long time Strichartz estimate and its proof}

In contrast to the mass-critical regime, we focus on the low frequency instead of high frequency, and define that for any $G_k^j \subset [a,b]$, 
\begin{align*}
\norm{u}_{X(G_k^j \times \R^2)}^2 := & \sum_{i : 0 \leq i <j}  2^{i-j} \sum_{G_{\alpha}^i \subset G_k^j} \norm{\abs{\nabla}^{\frac{1}{2}}  P_{ 2^{-i}} u}_{U_{\Delta}^2 (G_{\alpha}^i \times \R^2)}^2 + \sum_{i : i \geq j}  \norm{\abs{\nabla}^{\frac{1}{2}} P_{2^{-i}} u}_{U_{\Delta}^2 (G_{k}^j \times \R^2)}^2 .
\end{align*}
Then define
\begin{equation*}
\norm{u}_{\tilde{X}_{k_0} ([a,b] \times \R^2)}^2 : = \sup_{j: 0 \leq j \leq k_0} \sup_{G_k^j \subset [a,b]} \norm{u}_{X(G_k^j \times \R^2)}^2 .
\end{equation*}
Similarly, we want to show $\norm{u}_{\tilde{X}_{k_0} ([0,T] \times \R^2)} \lesssim 1$, which captures the essential feature in the case that if we assume the double endpoint were true:
\begin{equation*}
\norm{\abs{\nabla}^{\frac{1}{2}}  P_{< N} u}_{L_t^2 L_x^{\infty} (J \times \R^2)} \lesssim \parenthese{KN}^{\frac{1}{2}} +1,
\end{equation*}
where $J$ is an interval satisfying 
\begin{equation*}
\int_J N(t) \, dt = K .
\end{equation*}

More precisely, we want to show that
\begin{thm}[Long time Strichartz estimate]\label{thm LTS}
If $u$ is an almost periodic solution to \eqref{NLS} then for any $M = 2^{k_0}$, $\varepsilon_1$, $\varepsilon_2$, $\varepsilon_3$ satisfying \eqref{eq Epsilon},  $\int_0^T N(t) \, dt =\varepsilon_3 M$, and $\int_0^T \int_{\R^2} \abs{u(t,x)}^8 \,dxdt = M$, we have
\begin{equation*}
\norm{u}_{\tilde{X}_{k_0} ([0,T] \times \R^2)}^2  \lesssim 1 .
\end{equation*}
\end{thm}

\begin{rmk}
Throughout this section the implicit constant depends only on $\norm{u}_{L_t^{\infty} \dot{H}_x^{\frac{1}{2}}}$, and not on $M$, or $\varepsilon_1$, $\varepsilon_2$, $\varepsilon_3$.
\end{rmk}

\begin{proof}[Proof of Theorem \ref{thm LTS}]

We want to show that for any $0 \leq j \leq k$ and $G_k^j \subset [0,T]$ by induction on $k_*$, 
\begin{equation*}
\sum_{i : 0 \leq i < j }  2^{i-j} \sum_{G_{\alpha}^i \subset G_k^j} \norm{\abs{\nabla}^{\frac{1}{2}}  P_{2^{-i}} u}_{U_{\Delta}^2 (G_{\alpha}^i \times \R^2)}^2 + \sum_{i : i \geq j}  \norm{\abs{\nabla}^{\frac{1}{2}} P_{2^{-i}} u}_{U_{\Delta}^2 (G_{k}^j \times \R^2)}^2 \lesssim 1 .
\end{equation*}

\subsubsection{Base case}
First, we start with the base case ($k_* =0$), that is, $\norm{u}_{\tilde{X}_0 ([0,T] \times \R^2)} \lesssim 1 $.

Let $J^{\alpha} =[a_{\alpha} , b_{\alpha}]$. By the integral equation, Strichartz estimates, duality (Proposition \ref{prop Duality}), $V_{\Delta}^2 \hookrightarrow U_{\Delta}^4 \hookrightarrow L_t^4 L_x^4 $ (Theorem \ref{prop Embed}), Definition \ref{defn Small intervals 2}, and Remark \ref{rmk Strichartz pair}, we write
\begin{align*}
\norm{  \abs{\nabla}^{\frac{1}{2}} u}_{U_{\Delta}^2 (J^{\alpha} \times \R^2)}  
& \lesssim \norm{ u }_{  L_t^{ \infty}\dot{H}^{\frac{1}{2}}(J^{\alpha} \times \R^2) }+ \sup_{\norm{v}_{V_{\Delta}^2}=1} \int_{J^{\alpha}} \inner{  v,  \abs{\nabla}^{\frac{1}{2}} F(u)} \, dt 
\lesssim 1 .
\end{align*}

To compute $\norm{u}_{X(J^{\alpha} \times \R^2)}$, we know that at the base case level, the only small interval inside a small interval $J^{\alpha}$ is itself, hence we have no first term in \eqref{X}. Then by Definition \ref{defn X}, Littlewood-Paley theorem, Minkowski, H\"older and Remark \ref{rmk Strichartz pair}, we obtain
\begin{align*}
\norm{u}_{X(J^{\alpha} \times \R^2)}^2 
& \lesssim \sum_{i: i \geq 0} \parenthese{ \norm{ \abs{\nabla}^{\frac{1}{2}}  P_{2^{-i}}  u (a_{\alpha})}_{L_x^2}^2 + \norm{  \abs{\nabla}^{\frac{1}{2}}  P_{2^{-i}}  F(u) }_{L_t^{1} L_x^{2}(J^{\alpha}\times \R^2) }^2 } 
 \lesssim 1  .
\end{align*}

Note that when $k_{\ast} =0$, the supremum of \eqref{X} over all intervals $G_k^j \subset [0,T]$ becomes the the supremum of over all small intervals $J^{\alpha}\subset [0,T]$. Therefore, by Definition \ref{defn X}, 
\begin{equation}\label{eq Base}
\norm{u}_{\tilde{X}_{0} ([0,T] \times \R^2)}^2  = \sup_{ j = 0} \sup_{J^{\alpha} \subset [0,T]} \norm{u}_{X(J^{\alpha} \times \R^2)}^2 \leq C(u)  .
\end{equation}
Notice that $C(u)$ only depends on $\norm{u}_{L_t^{\infty} \dot{H}^{\frac{1}{2}}[0,T ] \times \R^2}$.

\subsubsection{Induction}
By Definition \ref{defn X}, Lemma \ref{lem U^p sum}, we have
%
%
\begin{equation}\label{eq Induction}
\norm{u}_{\tilde{X}_{k_{\ast}+1 } ([0,T] \times \R^2)}^2 \leq 2 \norm{u}_{\tilde{X}_{k_{\ast}} ([0,T] \times \R^2)}^2  .
\end{equation}
Therefore, by \eqref{eq Base} and \eqref{eq Induction},
\begin{equation}\label{eq 11step}
\norm{u}_{\tilde{X}_{11 } ([0,T] \times \R^2)}^2  \leq 2^{11} C(u)  .
\end{equation}

\subsubsection{Bootstrap}
For $j>11$ and $G_k^j \subset [0,T]$, we want to prove that $\norm{u}_{\tilde{X}_{j} (G_k^j  \times \R^2)} \leq 2^{11} C(u)$ by bootstrap argument. 

First consider the terms in the $X$ norm in \eqref{X} with frequencies localized higher than $2^{-11}$, and they are bounded due to \eqref{eq 11step} (see Step 1). For the terms in the $X$ norm with frequencies localized lower than $2^{-11}$, we can use the integral equation to rewrite $u$ into the free solution and the Duhamel term, then compute the contributions of these two terms to the first term (A) and the second term (B) of the $X$ norm respectively. As a result, we have a free solution term (AF) and a Duhamel term (AD) contributing to A, and a free solution term (BF) and a Duhamel term (BD) contributing to B. 
\begin{equation*}
X(G_k^j) \text{ norm }
\begin{cases}
\text{frequencies higher than } 2^{-11} \text{ (Step 1)}\\
\text{frequencies lower than } 2^{-11}
\begin{cases}
\text{1st term } A  
\begin{cases} \text{free solution } AF \text{ (Step 2)}\\ \text{Duhamel term }  AD \text{ (Step 4)}\end{cases} \\
\text{2nd term }B
\begin{cases} \text{free solution } BF \text{ (Step 3)}\\ \text{Duhamel term } BD \text{ (Step 4)} \end{cases} 
\end{cases}
\end{cases}
\end{equation*}

We will consider the terms with frequencies higher than $2^{-11}$ in Step 1. And estimates the free solution terms in A and B in Step 2 and Step 3 respectively. In this proof, the hardest part is to estimate AD and BD. We will bound them in Step 4.

It is worth mentioning that in Step 4, we treat two different types of $G_{\alpha}^i$ intervals in two cases (see the classification of these two cases in Step 4). For case 1, we compute directly, while for case 2, we will prove a bootstrap argument Proposition \ref{prop Term 1&2}. In the proof we decompose the nonlinear term $\abs{u}^4 u$ into different frequencies and consider them in Lemma \ref{lem Term1} and Lemma \ref{lem Term2}.
\begin{equation*}
\text{Step 4 } (AD \, \& \, BD) \to 
\begin{cases}
\text{case 1}\\
\text{case 2} \to \text{Proposition \ref{prop Term 1&2}}
\begin{cases}
\text{Lemma \ref{lem Term1}}\\
\text{Lemma \ref{lem Term2}}
\end{cases}
\end{cases}
\end{equation*}

\noindent \underline{\bf Step 1: Frequencies higher than $2^{-11}$}

Note that $G_j^k$ overlaps $2^{j-11}$ intervals $G_{\beta}^{11}$ and $G_{\beta}^{11}$ overlaps $2^{11-i}$ intervals $G_{\alpha}^{i}$. So by Fubini-Tonelli theorem, \eqref{eq 11step} implies that for any $j > 11$ and $G_k^j \subset [0,T]$,
\begin{equation*}
\sum_{i: 0 \leq i \leq 11 } 2^{i-j} \sum_{G_{\alpha}^i \subset G_k^j} \norm{\abs{\nabla}^{\frac{1}{2}}  P_{ 2^{-i}} u}_{U_{\Delta}^2 (G_{\alpha}^i \times \R^2)}^2 \leq 2^{11}C(u) .
\end{equation*}


\noindent \underline{\bf Step 2: Free solution term in A}

Fix $k_0$, $12 \leq j \leq k_0$, and $G_k^j \subset [0,T]$. For $11 \leq i < j$, Duhamel's principle implies
\begin{equation*}
\norm{\abs{\nabla}^{\frac{1}{2}}  P_{ 2^{-i}} u}_{U_{\Delta}^2 (G_{\alpha}^i \times \R^2)} 
 \lesssim  \norm{\abs{\nabla}^{\frac{1}{2}}  P_{2^{-i}} u(t_{\alpha}^i)}_{L_x^2 (\R^2)}  + \norm{\abs{\nabla}^{\frac{1}{2}}  \int_{t_{\alpha}^i}^t  e^{i(t-t^{\prime})\Delta} P_{ 2^{-i}} F(u) \, dt^{\prime}  }_{U_{\Delta}^2 (G_{\alpha}^i \times \R^2)} .
\end{equation*}

\noindent Choose $t_{\alpha}^i$ satisfying
\begin{equation}\label{eq1 Thm LTS}
\norm{\abs{\nabla}^{\frac{1}{2}}  P_{ 2^{-i}} u(t_{\alpha}^i)}_{L_x^2(\R^2)}  = \inf_{t \in G_{\alpha}^i} \norm{\abs{\nabla}^{\frac{1}{2}}  P_{2^{-i}} u(t)}_{L_x^2(\R^2)}  .
\end{equation}

\noindent Then by Definition \ref{defn Small intervals 2}, Fubini-Tonelli theorem and \eqref{eq1 Thm LTS},
\begin{align*}
& \quad \sum_{i : 11 \leq i < j } 2^{i-j} \sum_{G_{\alpha}^i \subset G_k^j} \norm{\abs{\nabla}^{\frac{1}{2}}  P_{ 2^{-i-2} \leq \cdot \leq 2^{-i+2}} u(t_{\alpha}^i)}_{L_x^2 (\R^2)}^2 \\
&  \leq \frac{1}{2\varepsilon_3} 2^{-j} \sum_{G_{\alpha}^i \subset G_k^j}  \int_{G_{\alpha}^i} \sum_{i : 11 \leq i < j }  \norm{\abs{\nabla}^{\frac{1}{2}}  P_{ 2^{-i} } u(t)}_{L_x^2 (\R^2)}^2 \parenthese{\varepsilon_3 \norm{u(t)}_{L_x^8(\R^2)}^8 + N(t)}  \, dt \lesssim 1 .
\end{align*}

\noindent \underline{\bf Step 3: Free solution term in B}

For $i \geq j$ simply take $t_0$, where $t_0$ is a fixed element of $G_k^j$, say the left endpoint. Then
\begin{equation*}
\sum_{i : i \geq j } \norm{\abs{\nabla}^{\frac{1}{2}}  P_{ 2^{-i}} e^{it\Delta} u(t_0)}_{U_{\Delta}^2 (G_k^j \times \R^2)}^2 \lesssim \sum_{i : i \geq j } \norm{\abs{\nabla}^{\frac{1}{2}}  P_{ 2^{-i}} u(t_0)}_{L_x^2(\R^2)}^2    \lesssim 1 .
\end{equation*}

\noindent Therefore, from Step 2 and Step 3, we have the following bound for the free solution terms AF and BF:
\begin{equation*}
\sum_{i : 0 \leq i < j}  2^{i-j} \sum_{G_{\alpha}^i \subset G_k^j} \norm{\abs{\nabla}^{\frac{1}{2}}  P_{ 2^{-i}} u}_{L_x^2(\R^2)}^2 
+ \sum_{i : i \geq j}  \norm{\abs{\nabla}^{\frac{1}{2}} P_{2^{-i}} u}_{L_x^2(\R^2)}^2  \lesssim 1 .
\end{equation*}

Thanks to the calculation above, we have 
\begin{align*}
\norm{u}_{X(G_k^j \times \R^2)}^2 & \lesssim 1+ \sum_{i : 11 \leq i < j } 2^{i-j}  \sum_{G_{\alpha}^i \subset G_k^j} \norm{\abs{\nabla}^{\frac{1}{2}}  \int_{t_{\alpha}^i}^t  e^{i(t-t^{\prime})\Delta} P_{ 2^{-i}} F(u) \, dt^{\prime}  }_{U_{\Delta}^2 (G_{\alpha}^i \times \R^2)}^2 \\
& + \sum_{ i : i \geq j } \norm{\abs{\nabla}^{\frac{1}{2}}  \int_{t_0}^t  e^{i(t-t^{\prime})\Delta} P_{ 2^{-i}} F(u) \, dt^{\prime}  }_{U_{\Delta}^2 (G_k^j \times \R^2)}^2  .
\end{align*}

\noindent \underline{\bf Step 4: Duhamel terms in A and B}

For these two terms, we consider the intervals $G_{\alpha}^i$ and $G_k^j$ in the following cases:
\begin{equation*}
\begin{cases}
\text{Case 1 : } G_{\alpha}^i: G_{\alpha}^i \subset G_k^j \text{ and } N(G_{\alpha}^i ) \leq \varepsilon_3^{-1/2} 2^{-i} ; \text{ and } G_k^j: N(G_k^j ) \leq \varepsilon_3^{-1/2} 2^{-j} ,\\
\text{Case 2 : } G_{\alpha}^i : G_{\alpha}^i \subset G_k^j \text{ and }  N(G_{\alpha}^i ) \geq \varepsilon_3^{-1/2} 2^{-i}; \text{ and } G_k^j: N(G_k^j ) \geq \varepsilon_3^{-1/2} 2^{-j} .
\end{cases}
\end{equation*}

\noindent {\bf $\bullet$ Case 1 in Step 4:} 
There are at most two small intervals, call them $J_1$ and $J_2$, that intersect $G_k^j$ but are not contained in $G_k^j$. Therefore, by Minkowski, H\"older and Remark \ref{rmk Strichartz pair}
\begin{equation}\label{eq2 Thm LTS}
\begin{aligned}
\sum_{i : 11 \leq i < j } 2^{i-j}  \sum_{G_{\alpha}^i \subset G_k^j} \norm{\abs{\nabla}^{\frac{1}{2}} F(u) }_{L_t^1L_x^2(G_{\alpha}^i \cap (J_1 \cup J_2) \times \R^2)}^2 
& \lesssim \sum_{i : 11 \leq i < j} 2^{i-j}  \norm{\abs{\nabla}^{\frac{1}{2}} F(u) }_{L_t^1L_x^2(J_1 \cup J_2 \times \R^2)}^2 \lesssim 1 
\end{aligned}
\end{equation}

Next observe that $N(G_{\alpha}^i) \leq \varepsilon_3^{-1/2} 2^{-i}$ implies that $N(t) \leq \varepsilon_3^{-1/2} 2^{-i+1}$ for all $t \in G_{\alpha}^i$. In fact, by Remark \ref{rmk Difference 1/N(t)} the difference of $\frac{1}{N(t)}$ on $G_{\alpha}^i$ is at most $  \varepsilon_1^{-1/2} \varepsilon_3 2^{i}  $, hence
\begin{equation*}
N(t) \leq N(G_{\alpha}^i) + \varepsilon_1^{-1/2} \varepsilon_3 2^{i}   \leq   \varepsilon_3^{-1/2} 2^{-i+1} \quad \text{ for all } t \in G_{\alpha}^i  .
\end{equation*}

Now by Definition \ref{defn Small intervals 1} and Definition \ref{defn Small intervals 2}, the number of $J_l$ intervals inside $G_{\alpha}^i$ is
\begin{equation*}
\# \{J_l : J_l \subset G_{\alpha}^i\} \sim  \int_{\cup J_l} N(t)^2 \, dt  \lesssim  \int_{G_{\alpha}^i } N(t)^2 \, dt \leq  \varepsilon_3^{-1/2} 2^{-i+1} \int_{G_{\alpha}^i} N(t) \, dt \leq \varepsilon_3^{1/2} 2^{2}   .
\end{equation*}
This implies that the number of intervals $J_l$ such that $N(t) \leq \varepsilon_3^{-1/2} 2^{-i+1}$ for all $t \in G_{\alpha}^i$ is finite and does not depend on $G_{\alpha}^i$. By \eqref{eq2 Thm LTS}, Fubini-Tonelli theorem and Remark \ref{rmk N(J_l)}, we have
\begin{align*}
& \quad \sum_{i: 11 \leq i < j} 2^{i-j} \sum_{ \substack{G_{\alpha}^i \subset G_k^j \\ N(G_{\alpha}^i) \leq \varepsilon_3^{-1/2} 2^{-i}}}  \norm{\abs{\nabla}^{\frac{1}{2}}  \int_{t_{\alpha}^i}^t  e^{i(t-t^{\prime})\Delta} P_{ 2^{-i}} F(u) \, dt^{\prime}  }_{U_{\Delta}^2 (G_{\alpha}^i \times \R^2)}^2 \\
& \lesssim \sum_{i : 11 \leq i < j } 2^{i-j}  \sum_{G_{\alpha}^i \subset G_k^j} \norm{\abs{\nabla}^{\frac{1}{2}} F(u) }_{L_t^1L_x^2(G_{\alpha}^i \cap (J_1 \cup J_2) \times \R^2)}^2 \\
& + \sum_{i : 11 \leq i < j} 2^{i-j}  \sum_{ \substack{J_l \subset G_k^j \\ N(J_l) \leq \varepsilon_3^{-1/2} 2^{-i+1} } }  \norm{\abs{\nabla}^{\frac{1}{2}} F(u) }_{L_t^1L_x^2(J_l \times \R^2)}^2  \lesssim 1 
\end{align*}

\noindent Similarly, if $N(G_k^j) \leq \varepsilon_3^{-1/2} 2^{-j} $, then $N(t) \leq \varepsilon_3^{-1/2} 2^{-j+1}$ for all $t \in G_k^j$. This implies that 
\begin{equation*}
 \norm{ u  }_{L_t^{8} L_x^{8}(G_k^j  \times \R^2)}^8  \sim \int_{G_k^j} N(t)^2 \, dt \leq \varepsilon_3^{-1/2} 2^{-j+1} \int_{G_k^j} N(t) \, dt \leq \varepsilon_3^{-1/2} 2^{-j+1} \varepsilon_3  2^j \lesssim 1  .
\end{equation*}
Hence,
\begin{equation}\label{eq3 Thm LTS}
 \norm{ \abs{\nabla}^{\frac{1}{2}} u  }_{L_t^{q} L_x^{r}(G_k^j  \times \R^2)} \lesssim 1 \quad \text{ for any admissible pair } (q,r) .
\end{equation}

\noindent Therefore, by Minkowski and H\"older and \eqref{eq3 Thm LTS}
\begin{equation*}
\sum_{\substack{ i : i \geq j \\ N(G_k^j) \leq \varepsilon_3^{-1/2} 2^{-j} }} \norm{\abs{\nabla}^{\frac{1}{2}}  P_{ 2^{-i}} F(u)  }_{L_t^1L_x^2(G_k^j  \times \R^2)}^2  \lesssim \norm{ \abs{\nabla}^{\frac{1}{2}}  P_{ \leq 2^{-j+2}} F(u)  }_{L_t^1L_x^2 (G_k^j  \times \R^2) }^2  \lesssim   1  .
\end{equation*}

Therefore, all the computation above yields
\begin{align}\label{eq4 Thm LTS}
\norm{u}_{X(G_k^j \times \R^2)}^2 & \lesssim 1+ \sum_{i ; 11 \leq i < j } 2^{i-j}  \sum_{\substack{ G_{\alpha}^i \subset G_k^j \\ N(G_{\alpha}^i) \geq \varepsilon_3^{-1/2} 2^{-i} }} \norm{\abs{\nabla}^{\frac{1}{2}}  \int_{t_{\alpha}^i}^t  e^{i(t-t^{\prime})\Delta} P_{ 2^{-i}} F(u) \, dt^{\prime}  }_{U_{\Delta}^2 (G_{\alpha}^i \times \R^2)}^2 \notag\\
& + \sum_{ \substack{ i : i \geq j\\ N(G_k^j) \geq \varepsilon_3^{-1/2} 2^{-j}} }\norm{\abs{\nabla}^{\frac{1}{2}}  \int_{t_0}^t  e^{i(t-t^{\prime})\Delta} P_{ 2^{-i}} F(u)  \,  dt^{\prime}  }_{U_{\Delta}^2 (G_k^j \times \R^2)}^2  .
\end{align}

\noindent {\bf $\bullet$ Case 2 in Step 4:} 
From now on, we take the intervals $G_{\alpha}^i$ with $N(G_{\alpha}^i) \geq \varepsilon_3^{-1/2} 2^{-i}$ and the intervals $G_k^j$ with $N(G_k^j) \geq \varepsilon_3^{-1/2} 2^{-j}$.
Continue from \eqref{eq4 Thm LTS}. To close the proof of the long time Strichartz, it suffices to prove the following proposition:
\begin{prop}\label{prop Term 1&2}
\begin{equation}\label{eq Term 1&2}
\begin{aligned}
& \quad \sum_{i : 11 \leq i < j} 2^{i-j}  \sum_{\substack{ G_{\alpha}^i \subset G_k^j \\ N(G_{\alpha}^i) \geq \varepsilon_3^{-1/2} 2^{-i} }} \norm{\abs{\nabla}^{\frac{1}{2}}  \int_{t_{\alpha}^i}^t  e^{i(t-t^{\prime})\Delta} P_{ 2^{-i}} F(u) \,  dt^{\prime}  }_{U_{\Delta}^2 (G_{\alpha}^i \times \R^2)}^2 \\
& + \sum_{ \substack{ i: i \geq j \\ N(G_k^j) \geq \varepsilon_3^{-1/2} 2^{-j}} }\norm{\abs{\nabla}^{\frac{1}{2}}  \int_{t_0}^t  e^{i(t-t^{\prime})\Delta} P_{ 2^{-i}} F(u)  \,  dt^{\prime}  }_{U_{\Delta}^2 (G_k^j \times \R^2)}^2 \\
& \lesssim  \varepsilon_2^4 \norm{u}_{\tilde{X}_j([0,T] \times \R^2)}^6 + \varepsilon_2^2 \norm{u}_{\tilde{X}_j([0,T] \times \R^2)}^8 .
\end{aligned}
\end{equation}
\end{prop}

Indeed, assuming that Proposition \ref{prop Term 1&2} is true, we can run a bootstrap argument.

\noindent Suppose
\begin{equation}\label{eq Bootstrap}
\norm{u}_{\tilde{X}_{k_*} ([0,T] \times \R^2)}^2  \leq C_0 .
\end{equation}

\noindent Then by \eqref{eq Induction}
\begin{equation*}
\norm{u}_{\tilde{X}_{k_*+1} ([0,T] \times \R^2)}^2  \leq 2C_0 .
\end{equation*}

\noindent Then by \eqref{eq4 Thm LTS} and Proposition \ref{prop Term 1&2}, we have
\begin{equation*}
\norm{u}_{\tilde{X}_{k_*+1} ([0,T] \times \R^2)}^2  \leq C(u) \parenthese{1+ \varepsilon_2^{4} (2C_0)^{3} + \varepsilon_2^2 (2C_0)^4} .
\end{equation*}

\noindent Taking $C_0 = 2^{11}C(u)$ and $\varepsilon_2>0$ sufficiently small, it implies that \eqref{eq Bootstrap} holds for $k_* = 11$. The bootstrap argument is closed, since we obtain
\begin{equation*}
\norm{u}_{\tilde{X}_{k_*+1} ([0,T] \times \R^2)}^2  \leq C_0 .
\end{equation*}
Hence the long time Strichartz estimates follow by \eqref{eq Base}, \eqref{eq Induction} and induction on $k_*$.

Now we are left to show Proposition \ref{prop Term 1&2}.
\begin{proof}[Proof of Proposition \ref{prop Term 1&2}]
We first write $F(u)$ into the Littlewood-Paley decomposition. Without loss of generality, we assume that $2^{-n_1} \geq 2^{-n_2} \geq 2^{-n_3} \geq 2^{-n_4} \geq 2^{-n_5} $, i.e. $n_1 \leq  n_2  \leq n_3 \leq n_4  \leq n_5 $. In the proof, since we utilize Lebesgue norms $L_x^q L_x^r$ and Lebesgue norms do not see the difference between $u$ and $\bar{u}$, i.e, $\norm{P_N u}_{L_x^q L_x^r} = \norm{P_N \bar{u}}_{L_x^q L_x^r} $, it is safe for us to write
\begin{equation*}
P_{ 2^{-i}} F(u) \simeq P_{ 2^{-i}}   \parenthese{ \sum_{  n_1 \leq  n_2  \leq n_3 \leq n_4  \leq n_5  } f  }   
\end{equation*}
where $f =(P_{ 2^{-n_1}} u ) (P_{ 2^{-n_2}} u )(P_{ 2^{-n_3}} u )(P_{ 2^{-n_4}} u )(P_{ 2^{-n_5}} u )$.

Then we compare the largest frequency $2^{-n_1}$ with $2^{-i}$:
\begin{itemize}
\item
If $2^{-i} \gg 2^{-n_1}$, it is impossible since $P_{2^{-i}} ((P_{\leq 2^{-i-7}} u)^5)=0$;

\item
If $2^{-i} \sim 2^{-n_1}$(i.e. $2^{-n_1 -7} \leq 2^{-i} \leq 2^{-n_1 +7}$), then this is our fist case: 
\begin{equation*}
2^{-i} \sim 2^{-n_1} \geq 2^{-n_2} \geq 2^{-n_3} \geq 2^{-n_4} \geq 2^{-n_5} ;
\end{equation*}

\item
If $2^{-i} \ll 2^{-n_1}$, then $2^{-n_2}$ must have a similar frequency with $2^{-n_1}$, otherwise the projection to $2^{-i}$ frequency of this term will be zero. Hence the second case is 
\begin{equation*}
2^{-i} \leq  2^{-n_1} \sim 2^{-n_2} ; 2^{-n_1} \geq 2^{-n_2} \geq 2^{-n_3} \geq 2^{-n_4} \geq 2^{-n_5}  .
\end{equation*}
\end{itemize}

Therefore, it is sufficient to consider the following:
\begin{align*}
P_{ 2^{-i}} F(u) &  \simeq P_{ 2^{-i}}   \parenthese{ \sum_{  \substack{ n_1 \sim i ; \\ n_1 \leq  n_2  \leq n_3 \leq n_4  \leq n_5 } } f }   + P_{2^{-i}}   \parenthese{ \sum_{ \substack{  n_1  < i; n_1 \sim n_2 ; \\ n_1 \leq n_2 \leq n_3 \leq n_4  \leq n_5}} f  }  : = \tilde{F}_1 + \tilde{F}_2   .
\end{align*}

Next, we compute the contributions of $ \tilde{F}_1$ and $ \tilde{F}_2$ to \eqref{eq Term 1&2} in Lemma \ref{lem Term1} and Lemma \ref{lem Term2} respectively.

\begin{lem}[Contribution of $\tilde{F}_1$]\label{lem Term1}
For a fixed $G_k^j \subset [0,T]$, $j>11$,
\begin{align}
& \sum_{i : 11 \leq i <j } 2^{i-j}  \sum_{\substack{ G_{\alpha}^i \subset G_k^j \\ N(G_{\alpha}^i) \geq \varepsilon_3^{-1/2} 2^{-i} }} \norm{\abs{\nabla}^{\frac{1}{2}}  \int_{t_{\alpha}^i}^t  e^{i(t-t^{\prime})\Delta}    \tilde{F}_1   \,    dt^{\prime}  }_{U_{\Delta}^2 (G_{\alpha}^i \times \R^2)}^2 \label{eq Term1 lem term1}\\
& + \sum_{ \substack{ i : i \geq j \\ N(G_k^j) \geq \varepsilon_3^{-1/2} 2^{-j}} }\norm{\abs{\nabla}^{\frac{1}{2}}  \int_{t_0}^t  e^{i(t-t^{\prime})\Delta}   \tilde{F}_1  \,  dt^{\prime}  }_{U_{\Delta}^2 (G_k^j \times \R^2)}^2 \label{eq Term2 lem term1}\\
& \lesssim \varepsilon_2^4 \norm{u}_{\tilde{X}_j ([0,T]\times \R^2)}^6 .   \notag  
\end{align}

\end{lem}

\begin{proof}
By Proposition \ref{prop Duality},
\begin{align}
& \quad \norm{\abs{\nabla}^{\frac{1}{2}}  \int_{t_{\alpha}^i}^t  e^{i(t-t^{\prime})\Delta}   \tilde{F}_1   \,    dt^{\prime}  }_{U_{\Delta}^2 (G_{\alpha}^i \times \R^2)}  \lesssim \sup_{\norm{v}_{V_{\Delta}^2 (G_{\alpha}^i \times \R^2)} =1} \int_{G_{\alpha}^i}  \inner{v , \abs{\nabla}^{\frac{1}{2}}  \tilde{F}_1 }  \,  dt \notag\\
& \lesssim  \sum_{  \substack{ n_1 \sim i ; \\ n_1 \leq  n_2  \leq n_3 \leq n_4  \leq n_5  }}  \sup_{\norm{v}_{V_{\Delta}^2 (G_{\alpha}^i \times \R^2)} =1} \int_{G_{\alpha}^i}  \inner{\abs{\nabla}^{\frac{1}{2}} P_{ 2^{-i}}    v ,     f}   \,  dt  . \label{eq5 Thm LTS}   
\end{align}

Lemma \ref{lem Local bounded} implies that $N(t)$ is bounded on $G_{\alpha}^i$, and recall $N(G_{\alpha}^i) = \inf_{t \in G_{\alpha}^i} N(t)$, $ 2^{-n_5} \leq 2^{-n_4} \leq \varepsilon_3^{1/2} N(t)$ and \eqref{eq Epsilon}, we know that 
\begin{equation}\label{eq6 Thm LTS}
\begin{aligned}
\norm{P_{ 2^{-n_4}} u  }_{L_{t}^{\infty} L_{x}^{ 4} (G_{\alpha}^i \times \R^2) } \lesssim \varepsilon_2 , \quad
\norm{P_{ 2^{-n_5}} u  }_{L_{t}^{\infty} L_{x}^{ 4} (G_{\alpha}^i \times \R^2) } \lesssim \varepsilon_2 .
\end{aligned}
\end{equation}
Then using H\"older and Bernstein, we write
\begin{align}
\int_{G_{\alpha}^i}  \inner{\abs{\nabla}^{\frac{1}{2}} P_{ 2^{-i}} v , f} \, dt 
 \lesssim 2^{-\frac{i}{2}} \norm{  v}_{L_t^{4} L_x^{4} (G_{\alpha}^i \times \R^2) } \norm{P_{ 2^{-n_1}} u  }_{L_{t}^{4} L_{x}^{ 4} (G_{\alpha}^i \times \R^2) } \norm{P_{ 2^{-n_2}} u  }_{L_{t}^{4} L_{x}^{ 4} (G_{\alpha}^i \times \R^2) } \notag \\
\times  \norm{P_{ 2^{-n_3}} u  }_{L_{t}^{4} L_{x}^{ 4} (G_{\alpha}^i \times \R^2) } \norm{P_{ 2^{-n_4}} u  }_{L_{t}^{\infty} L_{x}^{ \infty} (G_{\alpha}^i \times \R^2) } \norm{P_{ 2^{-n_5}} u  }_{L_{t}^{\infty} L_{x}^{ \infty} (G_{\alpha}^i \times \R^2) } .  \label{eq7 Thm LTS}
\end{align}
By $V_{\Delta}^2 \hookrightarrow U_{\Delta}^4$ (Theorem \ref{prop Embed}), Bernstein, \eqref{eq6 Thm LTS} and \eqref{eq7 Thm LTS},
then \eqref{eq5 Thm LTS} becomes
\begin{align}
\eqref{eq5 Thm LTS} & \lesssim \sum_{  \substack{ n_1 \sim i ; \\ n_1 \leq  n_2  \leq n_3 \leq n_4  \leq n_5  }}  2^{\frac{n_1}{2} -\frac{i}{2} + \frac{n_2}{2} + \frac{n_3}{2} -\frac{n_4}{2} -\frac{n_5}{2}}  \norm{\abs{\nabla}^{\frac{1}{2}}  P_{ 2^{-n_1}} u  }_{U_{\Delta}^2 (G_{\alpha}^i \times \R^2) }  \notag\\
& \qquad \times \norm{\abs{\nabla}^{\frac{1}{2}}  P_{ 2^{-n_2}} u  }_{U_{\Delta}^2 (G_{\alpha}^i \times \R^2) } \norm{\abs{\nabla}^{\frac{1}{2}}  P_{ 2^{-n_3}} u  }_{U_{\Delta}^2 (G_{\alpha}^i \times \R^2) } \varepsilon_2^2 . \label{eq7.1 Thm LTS}
\end{align}

The sum over $n_2$, $n_3$, $n_4$ and $n_5$ in \eqref{eq7.1 Thm LTS} of the component depending only on these frequencies is bounded by a multiple of $\norm{u }_{\tilde{X}_j([0,T] \times \R^2)}^2 $. The sum over $n_1$ in \eqref{eq7.1 Thm LTS} is in fact a finite sum. 
Combining \eqref{eq5 Thm LTS}, \eqref{eq7.1 Thm LTS}, Definition \ref{defn X} and the fact that the number of $G_{\alpha}^i$ such that $G_{\alpha}^i \subset G_k^j$ is at most $2^7$ , we have
\begin{align*}
\eqref{eq Term1 lem term1} 
&  \lesssim \sum_{i: 11 \leq i < j} 2^{i-j} \sum_{\substack{ G_{\alpha}^i \subset G_k^j \\ N(G_{\alpha}^i) \geq \varepsilon_3^{-1/2} 2^{-i} }} \sum_{n_1 : n_1 \sim i } 2^{n_1-i} \norm{\abs{\nabla}^{\frac{1}{2}}  P_{ 2^{-n_1}} u  }_{U_{\Delta}^2 (G_{\alpha}^i \times \R^2) }^2 \norm{u }_{\tilde{X}_j([0,T] \times \R^2)}^4 \varepsilon_2^4 \\
& \lesssim \varepsilon_2^4 \norm{u }_{\tilde{X}_j([0,T] \times \R^2)}^6
\end{align*}

Similarly, the second term \eqref{eq Term2 lem term1} in Lemma \ref{lem Term1} becomes,
\begin{align*}
\eqref{eq Term2 lem term1} & \lesssim \sum_{ \substack{ i : i \geq j \\ N(G_k^j) \geq \varepsilon_3^{-1/2} 2^{-j}}}  \norm{\abs{\nabla}^{\frac{1}{2}} P_{  2^{-i}} u }_{U_{\Delta}^2 (G_k^j \times \R^2) }^2  \varepsilon_3^4 \norm{u }_{\tilde{X}_j([0,T] \times \R^2)}^4  \lesssim \varepsilon_2^4 \norm{u}_{\tilde{X}_j ([0,T]\times \R^2)}^6   .
\end{align*}

Therefore, the proof of Lemma \ref{lem Term1} is complete.
\end{proof}

Next, we estimate the contribution of $\tilde{F}_2$ in the decomposition in Proposition \ref{prop Term 1&2}.
\begin{lem}[Contribution of $\tilde{F}_2$]\label{lem Term2}
For a fixed $G_k^j \subset [0,T]$, $j>11$,
\begin{align}
& \quad \sum_{i : 11 \leq i < j } 2^{i-j}  \sum_{\substack{ G_{\alpha}^i \subset G_k^j \\ N(G_{\alpha}^i) \geq \varepsilon_3^{-1/2} 2^{-i} }} \norm{\abs{\nabla}^{\frac{1}{2}}  \int_{t_{\alpha}^i}^t  e^{i(t-t^{\prime})\Delta} \tilde{F}_2    \,  dt^{\prime}  }_{U_{\Delta}^2 (G_{\alpha}^i \times \R^2)}^2 \label{eq Term1 lem term2}\\
& + \sum_{ \substack{ i: i\geq j \\ N(G_k^j) \geq \varepsilon_3^{-1/2} 2^{-j}} }\norm{\abs{\nabla}^{\frac{1}{2}}  \int_{t_0}^t  e^{i(t-t^{\prime})\Delta}  \tilde{F}_2   \,  dt^{\prime}  }_{U_{\Delta}^2 (G_k^j \times \R^2)}^2 \label{eq Term2 lem term2} \\
& \lesssim  \varepsilon_2^2 \norm{u}_{\tilde{X}_j([0,T] \times \R^2)}^8 .  \notag
\end{align}
\end{lem}

\begin{proof}
First, we consider \eqref{eq Term1 lem term2}. By Lemma \ref{lem Decomp}, we decompose  
\begin{align}
& \quad \norm{\abs{\nabla}^{\frac{1}{2}}  \int_{t_{\alpha}^i}^t  e^{i(t-t^{\prime})\Delta} \tilde{F}_2  \,    dt^{\prime}  }_{U_{\Delta}^2 (G_{\alpha}^i \times \R^2)} \notag\\
& \leq \sum_{n_1 : 0 \leq n_1 \leq i  }  \sum_{\substack{n_2 , n_3, n_4 , n_5  : \\ n_1 \sim  n_2 \leq n_3 \leq  n_4 \leq n_5 }} \sum_{ G_{\beta}^{n_1} \subset G_{\alpha}^i }  \sup_{\norm{v}_{V_{\Delta}^2(G_{\beta}^{n_1} \times \R^2) }=1}    \int_{G_{\beta}^{n_1}} \inner{\abs{\nabla}^{\frac{1}{2}} P_{ 2^{-i}}  v ,  f  } \, dt . \label{eq Term1 decomp Thm LTS}
\end{align}

By H\"older,
\begin{equation}\label{eq8 Thm LTS}
\begin{aligned}
\int_{G_{\beta}^{n_1}} \inner{\abs{\nabla}^{\frac{1}{2}} P_{2^{-i}}v ,   f  } \, dt  \lesssim \norm{ \abs{\nabla}^{\frac{1}{2}} P_{2^{-i}} v}_{L_t^{\infty} L_x^{\infty} (G_{\beta}^{n_1} \times \R^2) } \norm{ \parenthese{P_{2^{-n_1}} u} \parenthese{P_{2^{-n_3}} u}}_{L_t^{2} L_x^{2} (G_{\beta}^{n_1} \times \R^2)} \\
\times \norm{\parenthese{P_{2^{-n_2}} u } \parenthese{P_{2^{-n_4}} u} }_{L_t^{2} L_x^{2} (G_{\beta}^{n_1} \times \R^2) }  \norm{ P_{2^{-n_5}} u}_{L_t^{\infty} L_x^{\infty} (G_{\beta}^{n_1} \times \R^2)}   .
\end{aligned}
\end{equation}
By Bernstein, Proposition \ref{prop Embed} and $\norm{v}_{V_{\Delta}^2 (G_{\beta}^{n_1} \times \R^2)} =1$,
\begin{equation*}
\norm{ \abs{\nabla}^{\frac{1}{2}} P_{2^{-i}} v}_{L_t^{\infty} L_x^{\infty} (G_{\beta}^{n_1} \times \R^2) }  \lesssim 2^{-\frac{3i}{2}} \norm{ P_{2^{-i}} v}_{L_t^{\infty} L_x^{2} (G_{\beta}^{n_1} \times \R^2) }   \lesssim 2^{-\frac{3i}{2}}   .
\end{equation*}

Next, we employ Bourgain's bilinear estimates (Lemma \ref{lem Bilinear}) and Bernstein to obtain the following bounds:
\begin{align*}
\norm{ \parenthese{P_{2^{-n_1}} u} \parenthese{P_{2^{-n_3}} u}}_{L_t^{2} L_x^{2} (G_{\beta}^{n_1} \times \R^2)}  \lesssim 2^{n_1} \norm{  \abs{\nabla}^{\frac{1}{2}} P_{ 2^{-n_1}} u }_{U_{\Delta}^2 (G_{\beta}^{n_1} \times \R^2) } \norm{  \abs{\nabla}^{\frac{1}{2}} P_{ 2^{-n_3}} u }_{U_{\Delta}^2 (G_{\beta}^{n_1} \times \R^2) }  , \\
 \norm{ \parenthese{P_{2^{-n_2}} u} \parenthese{P_{2^{-n_4}} u}}_{L_t^{2} L_x^{2} (G_{\beta}^{n_1} \times \R^2)}  \lesssim 2^{n_2} \norm{  \abs{\nabla}^{\frac{1}{2}} P_{ 2^{-n_2}} u }_{U_{\Delta}^2 (G_{\beta}^{n_1} \times \R^2) } \norm{  \abs{\nabla}^{\frac{1}{2}} P_{ 2^{-n_4}} u }_{U_{\Delta}^2 (G_{\beta}^{n_1} \times \R^2) } .
\end{align*}

For $P_{2^{-n_5}} u$ term, we treat $2^{-n_5} \leq \varepsilon_3^{1/4} N(t)$ and $2^{-n_5} \geq \varepsilon_3^{1/4} N(t)$ separately:
\begin{itemize}
\item

If $2^{-n_5} \leq \varepsilon_3^{1/4} N(t)$, then by Bernstein and \eqref{eq Epsilon}, 
\begin{equation*}
\norm{  P_{2^{-n_5}} u}_{L_t^{\infty} L_x^{\infty} (G_{\beta}^{n_1} \times \R^2) } \lesssim  2^{-\frac{n_5}{2}} \norm{  P_{2^{-n_5}} u}_{L_t^{\infty} L_x^{4} (G_{\beta}^{n_1} \times \R^2) } \varepsilon_2   \lesssim  2^{-\frac{n_5}{2}} \varepsilon_2  .
\end{equation*}

\item
If $2^{-n_5}  \geq \varepsilon_3^{1/4} N(t) $, then the assumption $N(t) \geq \varepsilon_3^{-1/2} 2^{-i}$ in {\bf Case 2}  implies $2^{n_5 } \leq \varepsilon_3^{1/4} 2^i$. Then by Bernstein, we have
\begin{equation*}
\norm{  P_{2^{-n_5}} u}_{L_t^{\infty} L_x^{\infty} (G_{\beta}^{n_1} \times \R^2) } \lesssim 2^{-\frac{n_5}{2}} \norm{ P_{2^{-n_5}} u}_{L_t^{\infty} L_x^{4} (G_{\beta}^{n_1} \times \R^2) } \lesssim 2^{\frac{n_5}{6}}2^{-\frac{2n_5}{3}} \leq  \varepsilon_3^{1/24} 2^{\frac{i}{6}-\frac{2n_5}{3}}   .
\end{equation*}

\end{itemize}

Using \eqref{eq Epsilon} again, we have the following bound for $P_{2^{-n_5}} u$ term,
\begin{equation*}
\norm{  P_{2^{-n_5}} u}_{L_t^{\infty} L_x^{\infty} (G_{\beta}^{n_1} \times \R^2) } \lesssim \varepsilon_2 2^{-\frac{n_5}{2}} \parenthese{1+ 2^{\frac{1}{6}(i-n_5)} } .
\end{equation*}

Putting the computations above together, combining with $\norm{v}_{V_{\Delta}^2 (G_{\beta}^{n_1} \times \R^2)} =1$ and \eqref{eq Epsilon}, we obtain
\begin{align}
\eqref{eq8 Thm LTS} &  \lesssim  2^{n_1 + n_2 -\frac{n_5}{2} -\frac{3i}{2}}   \varepsilon_2 \norm{  \abs{\nabla}^{\frac{1}{2}} P_{ 2^{-n_1}} u }_{U_{\Delta}^2  (G_{\beta}^{n_1} \times \R^2) } \norm{  \abs{\nabla}^{\frac{1}{2}} P_{ 2^{-n_2}} u }_{U_{\Delta}^2  (G_{\beta}^{n_1} \times \R^2) } \notag\\
& \qquad \times \norm{  \abs{\nabla}^{\frac{1}{2}} P_{ 2^{-n_3}} u }_{U_{\Delta}^2 (G_{\beta}^{n_1} \times \R^2) }  \norm{  \abs{\nabla}^{\frac{1}{2}} P_{ 2^{-n_4}} u }_{U_{\Delta}^2 (G_{\beta}^{n_1} \times \R^2) }  \label{eq Term1 in sum}\\
&  +  2^{n_1 + n_2 -\frac{n_5}{2} -\frac{3i}{2}}  2^{\frac{1}{6}(i-n_5)}  \varepsilon_2 \norm{  \abs{\nabla}^{\frac{1}{2}} P_{ 2^{-n_1}} u }_{U_{\Delta}^2  (G_{\beta}^{n_1} \times \R^2) } \norm{  \abs{\nabla}^{\frac{1}{2}} P_{ 2^{-n_2}} u }_{U_{\Delta}^2  (G_{\beta}^{n_1} \times \R^2) } \notag\\
& \qquad \times \norm{  \abs{\nabla}^{\frac{1}{2}} P_{ 2^{-n_3}} u }_{U_{\Delta}^2 (G_{\beta}^{n_1} \times \R^2) }  \norm{  \abs{\nabla}^{\frac{1}{2}} P_{ 2^{-n_4}} u }_{U_{\Delta}^2 (G_{\beta}^{n_1} \times \R^2) } . \label{eq Term2 in sum} 
\end{align}

Next, we consider the summations in \eqref{eq Term1 decomp Thm LTS}.

The sum over $n_3$, $n_4$ and $n_5$ acting on the first term \eqref{eq Term1 in sum} and the component depending on these frequencies is bounded by a multiple of $2^{\frac{3n_1}{2}  - \frac{3i}{2}}  \norm{u}_{\tilde{X}_j([0,T] \times \R^2)}^2 $. Then the sum over $n_2$ in \eqref{eq Term1 decomp Thm LTS} is a finite sum. So by Fubini-Tonelli theorem, Cauchy-Schwarz, Definition \ref{defn X}, \eqref{eq8 Thm LTS}, \eqref{eq Term1 in sum}, \eqref{eq Term2 in sum} and the fact that $G_{\beta}^{n_1}$ overlaps $2^{i-n_1}$ intervals $G_{\alpha}^{i}$, we obtain
\begin{align}
\eqref{eq Term1 decomp Thm LTS} 
& \lesssim \varepsilon_2 \norm{u}_{\tilde{X}_j([0,T] \times \R^2)}^3  \parenthese{\sum_{n_1 : 0 \leq n_1 \leq i  }  \sum_{ G_{\beta}^{n_1} \subset G_{\alpha}^i} 2^{(n_1 -i)\frac{5}{3}-}   \norm{  \abs{\nabla}^{\frac{1}{2}} P_{ 2^{-n_1}} u }_{U_{\Delta}^2 (G_{\beta}^{n_1} \times \R^2) }^2 }^{\frac{1}{2}}  .  \label{eq9 Thm LTS}
\end{align}

Therefore, by Cauchy-Schwarz, Fubini-Tonelli theorem, Definition \ref{defn X}, \eqref{eq Term1 decomp Thm LTS} and \eqref{eq9 Thm LTS}, we obtain
\begin{align}
\eqref{eq Term1 lem term2} & 
\lesssim  \varepsilon_2^2 \norm{u}_{\tilde{X}_j([0,T] \times \R^2)}^6 \sum_{n_1 : 0 \leq n_1 <j} 2^{n_1 -j } \sum_{i : n_1 \leq i <j } 2^{(n_1 -i)\frac{2}{3}-}   \sum_{ G_{\beta}^{n_1} \subset G_{\alpha}^i}   \norm{  \abs{\nabla}^{\frac{1}{2}} P_{ 2^{-n_1}} u }_{U_{\Delta}^2 (G_{\beta}^{n_1} \times \R^2) }^2  \notag\\
& \lesssim  \varepsilon_2^2 \norm{u}_{\tilde{X}_j([0,T] \times \R^2)}^8  .  \label{eq10 Thm LTS}
\end{align}

Then take \eqref{eq Term2 lem term2}.
Again by Lemma \ref{lem Decomp}, we decompose  
\begin{align}
& \quad \norm{\abs{\nabla}^{\frac{1}{2}}  \int_{t_{\alpha}^i}^t  e^{i(t-t^{\prime})\Delta} \tilde{F}_2  \,    dt^{\prime}  }_{U_{\Delta}^2 (G_{k}^j \times \R^2)} \notag\\
& \lesssim \sum_{n_1 : 0 \leq n_1 \leq j  }  \sum_{\substack{n_2 , n_3, n_4 ,n_5  : \\n_1 \sim  n_2 \leq n_3 \leq  n_4 \leq n_5 }} \sum_{ G_{\beta}^{n_1} \subset G_{k}^j }  \sup_{\norm{v}_{V_{\Delta}^2(G_{\beta}^{n_1} \times \R^2) }=1}    \int_{G_{\beta}^{n_1}} \inner{\abs{\nabla}^{\frac{1}{2}} P_{ 2^{-i}} v ,   f  } \, dt \label{eq Term2 decomp Thm LTS}\\
& + \sum_{n_1 : j \leq n_1 \leq i  }  \sum_{\substack{n_2 , n_3, n_4, n_5  : \\ n_1 \sim n_2 \leq n_3 \leq  n_4 \leq n_5}}  \sup_{\norm{v}_{V_{\Delta}^2(G_{k}^j \times \R^2) }=1}    \int_{G_{k}^j} \inner{\abs{\nabla}^{\frac{1}{2}} P_{ 2^{-i}} v ,   f  } \, dt .  \label{eq Term3 decomp Thm LTS}
\end{align}
Next we estimate these two terms separately.

By the same calculation from \eqref{eq8 Thm LTS} to \eqref{eq9 Thm LTS}, we have
\begin{align*}
\eqref{eq Term2 decomp Thm LTS} &\lesssim \varepsilon_2 \norm{u}_{\tilde{X}_j([0,T] \times \R^2)}^3  \parenthese{\sum_{n_1 : 0 \leq n_1 \leq j  }  \sum_{ G_{\beta}^{n_1} \subset G_{k}^j} 2^{(n_1 -i)\frac{5}{3}-}   \norm{  \abs{\nabla}^{\frac{1}{2}} P_{ 2^{-n_1}} u }_{U_{\Delta}^2 (G_{\beta}^{n_1} \times \R^2) }^2 }^{\frac{1}{2}} ,\\
\eqref{eq Term3 decomp Thm LTS}&  \lesssim \varepsilon_2 \norm{u}_{\tilde{X}_j([0,T] \times \R^2)}^3  \parenthese{\sum_{n_1 : j \leq n_1 \leq i  }  2^{(n_1 -i)\frac{5}{3}-}   \norm{  \abs{\nabla}^{\frac{1}{2}} P_{ 2^{-n_1}} u }_{U_{\Delta}^2 (G_{k}^{j} \times \R^2) }^2 }^{\frac{1}{2}} .
\end{align*}

Then by Fubini-Tonelli theorem and Definition \ref{defn X} and the similar calculation as in \eqref{eq10 Thm LTS}, we have
\begin{align*}
& \quad \eqref{eq Term2 lem term2} \\
& \lesssim  \varepsilon_2^2 \norm{u}_{\tilde{X}_j([0,T] \times \R^2)}^6 \sum_{ \substack{ i: i\geq j \\ N(G_k^j) \geq \varepsilon_3^{-1/2} 2^{-j}} }   \sum_{n_1 : 0 \leq n_1 \leq j  }  \sum_{ G_{\beta}^{n_1} \subset G_{k}^j} 2^{(n_1 -i)\frac{5}{3}-}   \norm{  \abs{\nabla}^{\frac{1}{2}} P_{ 2^{-n_1}} u }_{U_{\Delta}^2 (G_{\beta}^{n_1} \times \R^2) }^2  \\
& + \varepsilon_2^2 \norm{u}_{\tilde{X}_j([0,T] \times \R^2)}^6 \sum_{ \substack{ i: i\geq j \\ N(G_k^j) \geq \varepsilon_3^{-1/2} 2^{-j}} }   \sum_{n_1 : j \leq n_1 \leq i  }  2^{(n_1 -i)\frac{5}{3}-}   \norm{  \abs{\nabla}^{\frac{1}{2}} P_{ 2^{-n_1}} u }_{U_{\Delta}^2 (G_k^j \times \R^2) }^2  \\
& \lesssim \varepsilon_2^2 \norm{u}_{\tilde{X}_j([0,T] \times \R^2)}^8 .
\end{align*}
Therefore, we complete the proof of Lemma \ref{lem Term2}.
\end{proof}

Then Proposition \ref{prop Term 1&2} follows from Lemma \ref{lem Term1} and Lemma \ref{lem Term2}.  
\end{proof}
Now the proof of Theorem \ref{thm LTS} is complete.
\end{proof}

\begin{rmk}[Main differences with
 \cite{D3}]
After using Littlewood-Paley to decompose the nonlinearity in the Duhamel term, we should be very careful with the high frequency and high frequency interaction into low frequency terms (the worst case is five high frequencies interaction into low frequency). The reason here is that instead of proving Theorem \ref{thm LTS} directly, we are doing a bootstrap argument, that is, we wish to prove 
\begin{equation*}
\norm{u}_{\tilde{X}_{k_0} ([0,T] \times \R^2)}^2 \lesssim 1 + \varepsilon \norm{u}_{\tilde{X}_{k_0} ([0,T] \times \R^2)}^2 ,
\end{equation*}
then it is necessary to control some components of the left-hand side by some small number times itself. From the construction of the atomic $X$ norm, we can see that the high frequency terms require more summability than the others. Therefore, we should gain more decay than the mass-critical case to sum over the high frequency terms, and hence close the bootstrap argument as desired. In contrast, these terms were not problematic in mass-critical \cite{D3}, because the cutoff in the mass-critical problem and the cutoff in $\dot{H}^{\frac{1}{2}}$ are opposite, hence the worse case was all low frequencies interaction into high frequency. However, this case never happens since the contribution of all low frequencies remains low. 
\end{rmk}

\section{Impossibility of quasi-solition solutions}\label{sec No quasi-soliton}
After proving a suitable long time Strichartz estimate in Section \ref{sec LTS}. We now, in this section, are able to prove a frequency-localized interaction Morawetz estimate and use it to preclude the existence of quasi-soliton solutions.

\subsection{Interaction Morawetz estimate in 2D}\label{sec ME}
We first recall the interaction Morawetz estimate in dimensions two, with modified nonlinear terms, that is, we consider equations 
\begin{equation*}
 i\partial_t w + \Delta w =\abs{w}^4 w+ \NN ,
\end{equation*}
instead of 
\begin{equation*}
i\partial_t w + \Delta w =\abs{w}^4 w .
\end{equation*}
in \cite{PV}. We will use the following interaction Morawetz estimate to derive a frequency-localized interaction Morawetz estimate in the next subsection.

\begin{thm}\label{thm ME}
If $w$ solves the same equation
\begin{equation*}
i \partial_t w + \Delta w =  \abs{w}^4w + \NN  ,
\end{equation*}
then 
\begin{align*}
\frac{d}{dt} M_{\omega} (t)  & =4 \iint_{x_{\omega} = y_{\omega}} \abs{\partial_{\omega} (\overline{w(t,y)} w(t,x))}^2  \, dxdy  +  \frac{4}{3}\iint_{x_{\omega} = y_{\omega}} \abs{w(t,x)}^2 \abs{w(t,y)}^6  \, dxdy  \\
& \quad + 4\iint \abs{w(t,x)}^2  \frac{(x-y)_{\omega}}{\abs{(x-y)_{\omega}}} \re[ \NN \partial_{\omega}\bar{ w} -w \partial_{\omega}\bar{ \NN} ](t,y) \, dxdy \\
& \quad + 4  \iint \im [\bar{w} \partial_{\omega} w] (t,x) \frac{(x-y)_{\omega}}{\abs{(x-y)_{\omega}}}  \im [ \bar{w} \NN] (t,y) \, dxdy  .
\end{align*}
\end{thm}

\begin{lem}\label{lem H^1/2}
Define 
\begin{equation*}
M_y [w](t)= \int_{\R^2} \frac{x-y}{\abs{x-y}} \cdot \im [\bar{w} \nabla w](t,x)  \, dx ,
\end{equation*}
then
\begin{equation*}
\abs{M_y [w](t)} \lesssim \norm{w(t)}_{\dot{H}_x^{\frac{1}{2}}(\R^2)}^2  .
\end{equation*}
\end{lem}

\begin{proof}
This proof is followed from the proof of Lemma 6.9 in \cite{CKSTT3}.
\end{proof}

\subsection{Frequency-localized interaction Morawetz estimate}\label{sec FLME}
In this subsection, we prove a frequency-localized interaction Morawetz estimate and use it to preclude the existence of quasi-soliton solutions.
\begin{thm}[Frequency-localized interaction Morawetz estimate]\label{thm FLIME}
If $u$ is an almost periodic solution to \eqref{NLS} on $[0,T]$ with $\int_0^T N(t) \, dt =K$, then
\begin{equation}
\norm{ \abs{\nabla}^{\frac{1}{2}} \abs{P_{\geq \varepsilon_3 K^{-1}} u(t,x)}^2 }_{L_t^2 L_x^2 ([0,T] \times \R^2)} \lesssim o(K) ,
\end{equation}
where $o(K) $ is a quantity, such that $\frac{o(K)}{K} \to 0$ as $K \nearrow \infty$.
\end{thm}

\begin{proof}
Suppose $[0,T]$ is an interval such that for some integer $k_0$,
\begin{equation*}
\int_0^T \int_{\R^2} \abs{u(t,x)}^8 \, dxdt = 2^{k_0} .
\end{equation*}

\noindent Note that $\int_0^T N(t) \, dt =K$. Hence in order to apply Theorem \ref{thm LTS}, we need to do the scaling $u_{\lambda}(x)= \sqrt{\lambda} u( \lambda^2 t, \lambda x)$, where $\lambda = \frac{K}{\varepsilon_3 2^{k_0}}$. Since $\int N(t)^2 \, dt $ scales like $\norm{u}_{L_t^8 L_x^8}^8$, under the same scaling $N(t)$ should be $N_{\lambda} (t) = \lambda N (\lambda^2 t)$, therefore
\begin{equation*}
\int_0^{\frac{T}{\lambda^2}} N_{\lambda} (t) \, dt = \varepsilon_3 2^{k_0} .
\end{equation*}

\noindent Now we can apply Theorem \ref{thm LTS}, and have 
\begin{equation*}
\norm{u_{\lambda}}_{\tilde{X}_{k_0} ([0,\frac{T}{\lambda^2}] \times \R^2)} \lesssim 1 .
\end{equation*}
Note that in Theorem \ref{thm LTS} we only care about the low frequency component of the solution $u$, and already had a good upper bound for it. From now on, we will focus on the high frequency component of $u$.

Let $w = P_{\geq 2^{-k_0}} u_{\lambda}$, hence $w$ satisfies the following equation:
\begin{equation*}
i \partial_t w + \Delta w =P_{\geq 2^{-k_0}} F(u_{\lambda}) : = F(w) + \NN ,
\end{equation*}
where 
\begin{equation*}
\NN= P_{\geq 2^{-k_0}} F(u_{\lambda}) -F(w) .
\end{equation*}

\noindent Let 
\begin{equation*}
M_{\omega} (t)  = \iint_{\R^2 \times \R^2}  \abs{w(t,y)}^2 \frac{(x-y)_{\omega}}{\abs{(x-y)_{\omega}}} \cdot \im [\bar{w} \partial_{\omega} w] (t,x) \, dxdy .
\end{equation*}
By Corollary \ref{thm ME}, we get
\begin{align}\label{eq dM_omega}
\frac{d}{dt} M_{\omega} (t) & =4 \iint_{x_{\omega} = y_{\omega}} \abs{\partial_{\omega} (\overline{w(t,y)} w(t,x))}^2 \, dxdy  +  \frac{4}{3}\iint_{x_{\omega} = y_{\omega}} \abs{w(t,x)}^2 \abs{w(t,y)}^6 \, dxdy  \notag\\
& + 4\iint \abs{w(t,x)}^2  \frac{(x-y)_{\omega}}{\abs{(x-y)_{\omega}}} \re[ \NN \partial_{\omega}\bar{ w} -w \partial_{\omega}\bar{ \NN}](t,y) \, dxdy \\
& + 4  \iint \im [\bar{w} \partial_{\omega} w] (t,x) \frac{(x-y)_{\omega}}{\abs{(x-y)_{\omega}}}  \im [ \bar{w} \NN] (t,y) \, dxdy .  \notag
\end{align}

\noindent Recall $\omega \in S^1$, then we can write $\partial_\omega = \nabla \cdot \omega = \cos (\omega) \partial_1 + \sin (\omega) \partial_2$. Hence there exists a constant $C$ such that 
\begin{equation}\label{eq Int omega}
\int_{\omega \in S^1} \frac{x_{\omega}}{\abs{x_{\omega}}} (\nabla \cdot \omega) \, d\omega = C \frac{x}{\abs{x}} \cdot \nabla  .
\end{equation}
Therefore, 
\begin{equation}\label{eq M(t)}
M(t) = \int_{\omega \in S^1} M_{\omega} (t) \, d \omega = C \iint \abs{w(t,y)}^2 \frac{x-y}{\abs{x-y}} \im [\bar{w} \nabla w] (t,x) \, dxdy .
\end{equation}

Now we move the last two terms in \eqref{eq dM_omega} to the left hand side and integrate on both sides over $\omega$. The properties of the Radon transform in \cite{PV} imply:
\begin{align}\label{eq Radon}
& \iiint_{x_{\omega} = y_{\omega}} \abs{\partial_{\omega} (\overline{w(t,y)} w(t,x))}^2 \, dxdyd\omega  +\iiint_{x_{\omega} = y_{\omega}} \abs{w(t,y)}^2 \abs{w(t,x)}^6 \, dxdyd\omega  \notag \\
& \gtrsim \norm{ \abs{\nabla}^{\frac{1}{2}} \abs{w(t,x)}^2 }_{L_x^2 (\R^2)}^2 ,
\end{align} 
combining \eqref{eq Int omega}, 
we obtain
\begin{align}\label{eq dM(t)}
\norm{ \abs{\nabla}^{\frac{1}{2}} \abs{w(t,x)}^2 }_{L_x^2 (\R^2)}^2 & \lesssim \frac{d}{dt} M (t) \notag\\
& - \iint \abs{w(t,y)}^2   \frac{x-y}{\abs{x-y}}  \re[\NN  \nabla\bar{ w} -w\nabla \bar{ \NN} ](t,x) \, dxdy \\
& -   \iint \im [\bar{w} \nabla w] (t,y)  \frac{x-y}{\abs{x-y}}  \im [ \bar{w}\NN] (t,x) \, dxdy   \notag
\end{align}
where $M(t)$ is calculated in \eqref{eq M(t)}.

Then we integrate on both sides of \eqref{eq dM(t)} over time $t$, and the fundamental theorem of calculus in time yields,
\begin{align}
& \quad \norm{ \abs{\nabla}^{\frac{1}{2}} \abs{w(t,x)}^2 }_{L_t^2 L_x^2 ([0, \frac{T}{\lambda^2}] \times \R^2)}^2 \notag \\
& \lesssim \sup_{t \in [0, \frac{T}{\lambda^2}] } \abs{ \iint \abs{w(t,y)}^2 \frac{x-y}{\abs{x-y}} \im [\bar{w} \nabla w] (t,x) \, dxdy  } \label{eq M_1}\\ 
&+ \abs{\int_{0}^{\frac{T}{\lambda^2}} \iint \abs{w(t,y)}^2   \frac{x-y}{\abs{x-y}}  \re[\NN  \nabla\bar{ w} -w\nabla \bar{ \NN} ](t,x) \, dxdydt } \label{eq M_2} \\
& + \abs{ \int_{0}^{\frac{T}{\lambda^2}} \iint \im [\bar{w} \nabla w] (t,y)  \frac{x-y}{\abs{x-y}}  \im [ \bar{w}\NN] (t,x) \, dxdydt } . \label{eq M_3}
\end{align}

Next, we will estimate the terms \eqref{eq M_1} in Lemma \ref{lem M_1}, \eqref{eq M_2} in Lemma \ref{lem M_2} and \eqref{eq M_3} in Lemma \ref{lem M_3}. In the remainder of the proof all spacetime norms are over $[0, \frac{T}{\lambda^2}] \times \R^2$, unless indicated otherwise.

\begin{lem}\label{lem M_1}
There exists $\eta = \eta(K) > 0$ satisfying
\begin{equation*}
\eqref{eq M_1} \lesssim \eta 2^{k_0}.
\end{equation*}
\end{lem}

\begin{proof}
By Lemma \ref{lem H^1/2}, Bernstein and \eqref{eq Epsilon}, we obtain
\begin{equation*}
\eqref{eq M_1}  = \sup_{t \in [0, \frac{T}{\lambda^2}] } \abs{ \int \abs{w(t,y)}^2 \parenthese{\int \frac{x-y}{\abs{x-y}} \im [\bar{w} \nabla w] (t,x) \, dx} dy }  \lesssim \norm{w}_{L_t^{\infty} L_x^2}^2   \norm{w}_{L_t^{\infty} \dot{H}_x^{\frac{1}{2}}}^2   .
\end{equation*}
Now we claim 
\begin{equation}\label{eq w<eta}
\norm{w}_{L_t^{\infty} L_x^2}^2  \lesssim \eta 2^{k_0}.
\end{equation}
Assuming the claim is true, it is easy to see Lemma \ref{lem M_1} holds.

Then we are left to show the claim \eqref{eq w<eta}.
\begin{proof}[Proof of \eqref{eq w<eta}]
By Definition \ref{defn AP} and the fact that $N(t) \geq 1$, we know that for any $\eta$, there exists $c(\eta)$ such that
\begin{equation*}
\norm{\abs{\nabla}^{\frac{1}{2}} P_{\leq c(\eta) } u}_{L_t^{\infty} L_x^2}^2 \leq \eta .
\end{equation*}

\noindent Therefore under the scaling $u_{\lambda}(x)= \sqrt{\lambda} u( \lambda^2 t, \lambda x)$, $N_{\lambda} (t) = \lambda N (\lambda^2 t)$, we have $N_{\lambda}(t) \geq \lambda$, 
\begin{equation*}
\norm{\abs{\nabla}^{\frac{1}{2}} P_{\leq c(\eta) \lambda } u_{\lambda}}_{L_t^{\infty} L_x^2}^2 \leq \eta .
\end{equation*}
Now using Bernstein, Definition \ref{defn AP} and $\lambda= \frac{K}{\varepsilon_3 2^{k_0}}$, we can write
\begin{align*}
\norm{w}_{L_t^{\infty} L_x^2} 
&\lesssim 2^{\frac{k_0}{2}} \norm{\abs{\nabla}^{\frac{1}{2}} P_{2^{-k_0} < \cdot < c(\eta)\lambda} u_{\lambda}}_{L_t^{\infty} L_x^2} + c(\eta)^{-\frac{1}{2}}\lambda^{-\frac{1}{2}} \norm{\abs{\nabla}^{\frac{1}{2}} P_{> c(\eta)\lambda}u_{\lambda}}_{L_t^{\infty} L_x^2} \\
& \lesssim 2^{\frac{k_0}{2}}  \eta^{\frac{1}{2}} +  c(\eta)^{-\frac{1}{2}} \lambda^{-\frac{1}{2}}  = 2^{\frac{k_0}{2}}  \eta^{\frac{1}{2}} +  \parenthese{\frac{\varepsilon_3 2^{k_0}}{c(\eta)K}}^{\frac{1}{2}} .
\end{align*}
To prove \eqref{eq w<eta}, we only need to demand $\frac{\varepsilon_3}{c(\eta) K } < \eta$. And it is possible for us to choose suitable $\eta=\eta(K)$ such that $K c(\eta)\eta =1$, then we complete the proof of \eqref{eq w<eta}.
\end{proof}
\end{proof}

\begin{rmk}\label{rmk Interpolation}
For any admissible pair $(q,r)$, by H\"older, \eqref{eq w<eta} and Proposition \ref{prop X properties}, we can write
\begin{equation}\label{eq Interpolation}
\begin{aligned}
\norm{w}_{L_t^{q} L_x^{r}}  \lesssim \norm{w}_{L_t^{\infty} L_x^{2}}^{\frac{1}{r}}  \norm{w}_{L_t^{q(1-\frac{1}{r})} L_x^{2r-2}}^{1-\frac{1}{r}} & 
\lesssim \eta^{\frac{1}{2r}} 2^{\frac{k_0}{2}} \norm{u_{\lambda}}_{\tilde{X}_{k_0}([0,\frac{T}{\lambda^2}]\times \R^2)}^{1-\frac{1}{r}} .
\end{aligned}
\end{equation}
\end{rmk}

\begin{lem}\label{lem M_2}
There exists $\eta = \eta(K) > 0$ satisfying
\begin{equation*}
\eqref{eq M_2} \lesssim \eta 2^{k_0}  \parenthese{1+ \norm{u_{\lambda}}_{\tilde{X}_{k_0}([0,\frac{T}{\lambda^2}]\times \R^2)}}^6 .
\end{equation*}

\end{lem}

\begin{proof}
We first define the momentum bracket: 
\begin{equation*}
\bracket{ a,b}_p = \re [a \nabla \bar{b} - b \nabla \bar{a}].
\end{equation*}
Realizing that
\begin{equation*}
\bracket{ F(u) ,u}_p = \abs{u}^4 u \nabla \bar{u} - u \nabla (\abs{u}^4 \bar{u}) = - \abs{u}^2 \nabla (\abs{u}^4) = -\frac{2}{3} \nabla (\abs{u}^6) ,
\end{equation*}
we can rewrite the factor $\re[\NN  \nabla\bar{ w} -w\nabla \bar{ \NN} ]$ in \eqref{eq M_2} into
\begin{align*}
\bracket{ \NN, w}_p 
& = -\frac{2}{3} \nabla \parenthese{ \abs{u_{\lambda}}^6- \abs{u_{lo}}^6 -\abs{u_{hi}}^6}  -\bracket{ F(u_{\lambda}) -F(u_{lo}), u_{lo}}_p -\bracket{ P_{lo} F(u_{\lambda}), u_{hi}}_p 
\end{align*}
where $u_{hi} = P_{\geq 2^{-k_0}} u_{\lambda}= w$ and $u_{lo}=  P_{< 2^{-k_0}} u_{\lambda}$.

We obtain the contributions to \eqref{eq M_2} by integration by parts,
\begin{equation}\label{eq1 M_2}
\begin{aligned}
& \quad \sum_{i=1}^5 \int_{0}^{\frac{T}{\lambda^2}} \iint \abs{w(t,y)}^2   \frac{1}{\abs{x-y}}  \abs{u_{hi}(t,x)}^i \abs{u_{lo}(t,x)}^{6-i} \, dxdydt \\
& +  \int_{0}^{\frac{T}{\lambda^2}} \iint \abs{w(t,y)}^2   \frac{1}{\abs{x-y}}   \abs{u_{hi}(t,x)} \abs{P_{lo} F(u_{\lambda}(t,x))}  \, dxdydt \\
& + \sum_{i=1}^5  \int_{0}^{\frac{T}{\lambda^2}} \iint \abs{w(t,y)}^2  \abs{u_{hi}(t,x)}^i  \abs{u_{lo}(t,x)}^{5-i}  \abs{\nabla u_{lo}(t,x)}  \, dxdydt \\
& +  \int_{0}^{\frac{T}{\lambda^2}} \iint \abs{w(t,y)}^2  \abs{u_{hi}(t,x)} \abs{\nabla P_{lo} F (u_{\lambda}(t,x))}  \, dxdydt .
\end{aligned}
\end{equation}

Then using H\"older, Hardy-Littlewood-Sobolev and Bernstein, we have
\begin{align}
 \eqref{eq1 M_2} &  \lesssim \sum_{i=1}^4 \norm{w}_{L_t^{6} L_x^{3} }^2 \norm{u_{hi}^i u_{lo}^{6-i}}_{L_t^{\frac{3}{2}} L_x^{\frac{6}{5}} } \label{eq P_1}\\
&  +  \iint \abs{w(t,y)}^2   \frac{1}{\abs{x-y}}  \abs{u_{hi}(t,x)}^5\abs{u_{lo}(t,x)} \, dxdydt  \label{eq P_2}\\
& + \norm{w}_{L_t^{6} L_x^{3} }^2 \norm{u_{hi} P_{lo}(u_{hi}^5)}_{L_t^{\frac{3}{2}} L_x^{\frac{6}{5}} } \label{eq P_3}\\
&  + \sum_{i=1}^5  \norm{ w}_{L_t^{\infty} L_x^{2} }^2 \norm{u_{hi}^i u_{lo}^{5-i} \nabla u_{lo}}_{L_t^{1} L_x^{1} }\label{eq P_4}\\
&+ \sum_{i=0}^5 \norm{ w}_{L_t^{\infty} L_x^{2} }^2  \norm{u_{hi}\nabla P_{lo} \mathcal{O}(u_{hi}^i u_{lo}^{5-i})}_{L_t^{1} L_x^{1} } . \label{eq P_5}
\end{align}
Note that for \eqref{eq P_1}, \eqref{eq P_4} and \eqref{eq P_5}, it is sufficient to estimate the first and last summands in the terms.

\begin{enumerate}[(i)]
\item
First, consider \eqref{eq P_1}. By H\"older, Sobolev embedding, Bernstein and Proposition \ref{prop X properties}, we have
\begin{align*}
\norm{\mathcal{O}(u_{hi} u_{lo}^5)}_{L_t^{\frac{3}{2}} L_x^{\frac{6}{5}} }  & \lesssim \norm{u_{hi}}_{L_{t}^{3} L_{x}^{6}} \norm{u_{lo}}_{L_{t}^{15} L_{x}^{\frac{15}{2}} }^5  
\lesssim \norm{u_{\lambda}}_{\tilde{X}_{k_0}([0,\frac{T}{\lambda^2}]\times \R^2)} ^6\\
\norm{\mathcal{O}(u_{hi}^4 u_{lo}^2) }_{L_t^{\frac{3}{2}} L_x^{\frac{6}{5}} } & \lesssim \norm{u_{hi}}_{L_{t}^{\infty} L_{x}^{4}}^{2} \norm{u_{hi}}_{L_{t}^{3} L_{x}^{6}}^{2} \norm{u_{lo}}_{L_{t}^{\infty} L_{x}^{\infty}}^2 
\lesssim \norm{u_{\lambda}}_{\tilde{X}_{k_0}([0,\frac{T}{\lambda^2}]\times \R^2)}^4  .
\end{align*}

Then by Remark \ref{rmk Interpolation}, we have
\begin{equation*}
\eqref{eq P_1}  = \sum_{i=1}^4 \norm{w}_{L_t^{6} L_x^{3}  }^2 \norm{u_{hi}^i u_{lo}^{6-i}}_{L_t^{\frac{3}{2}} L_x^{\frac{6}{5}}  }  \lesssim \eta^{\frac{1}{3}} 2^{k_0}  \parenthese{1+\norm{u_{\lambda}}_{\tilde{X}_{k_0}([0,\frac{T}{\lambda^2}]\times \R^2)} }^{\frac{22}{3}} .
\end{equation*}

\item
Next, take \eqref{eq P_2}. In fact, we consider the following two scenarios: 
\begin{itemize}
\item
If $\abs{u_{lo}} \leq \delta \abs{u_{hi}}$ for some small $\delta >0$, this contribution will be absorbed into the following term
\begin{equation*}
\iiint_{x_{\omega} = y_{\omega}} \abs{w(t,x)}^2 \abs{w(t,y)}^6 dxdy d \omega  \simeq  \iint_{x_{\omega} = y_{\omega}} \frac{1}{\abs{x-y}} \abs{w(t,x)}^2 \abs{w(t,y)}^6 \,  dxdy  .
\end{equation*}

\item
If $\abs{u_{hi}} \leq \delta^{-1}\abs{u_{lo}} $, we can estimate the contribution of this term by 
\begin{equation*}
\int_{0}^{\frac{T}{\lambda^2}} \iint \abs{w(t,y)}^2   \frac{1}{\abs{x-y}}  \abs{u_{hi}(t,x)}^4 \abs{u_{lo}(t,x)}^2  \, dxdydt .
\end{equation*}

\end{itemize}

\item
Take \eqref{eq P_3} and apply H\"older, \eqref{eq Interpolation}, Bernstein and Proposition \ref{prop X properties}. 
\begin{align*}
\eqref{eq P_3} 
& \lesssim \eta^{\frac{1}{3}} 2^{k_0} \norm{u_{\lambda}}_{\tilde{X}_{k_0}([0,\frac{T}{\lambda^2}]\times \R^2)}^{\frac{4}{3}} 2^{-\frac{5k_0}{6}} \norm{u_{hi}}_{L_t^{\frac{5}{2}} L_x^{10}  }^{\frac{5}{3}} \norm{u_{hi}}_{L_t^{\infty} L_x^{4}  }^{\frac{10}{3}} \lesssim \eta^{\frac{1}{3}} 2^{k_0} \norm{u_{\lambda}}_{\tilde{X}_{k_0}([0,\frac{T}{\lambda^2}]\times \R^2)}^3  .
\end{align*}

\item
Take \eqref{eq P_4} and use H\"older, Sobolev embedding, Bernstein and Proposition \ref{prop X properties}
\begin{align*}
\norm{\mathcal{O} (u_{hi} u_{lo}^4) \nabla u_{lo}}_{L_t^{1} L_x^{1} } & \lesssim \norm{u_{hi} }_{L_t^{4} L_x^{4} }\norm{u_{lo}}_{L_t^{8} L_x^{8} }^4 2^{-\frac{k_0}{2}}\norm{ \abs{\nabla}^{\frac{1}{2}} u_{lo}}_{L_t^{4} L_x^{4} }   \lesssim \norm{u_{\lambda}}_{\tilde{X}_{k_0}([0,\frac{T}{\lambda^2}]\times \R^2)}^6 \\
\norm{\mathcal{O} (u_{hi}^5) \nabla u_{lo}}_{L_t^{1} L_x^{1}} & \lesssim \norm{u_{hi} }_{L_t^{3} L_x^{6} }^3 \norm{u_{hi} }_{L_t^{\infty} L_x^{4}}^2 2^{-\frac{k_0}{2}}\norm{ \abs{\nabla}^{\frac{1}{2}} u_{lo}}_{L_t^{\infty} L_x^{\infty}} 
\lesssim \norm{u_{\lambda}}_{\tilde{X}_{k_0}([0,\frac{T}{\lambda^2}]\times \R^2)}^3  .
\end{align*}

\noindent Therefore, put the calculations above together
\begin{equation*}
\eqref{eq P_4}  =  \sum_{i=1}^5  \norm{ w}_{L_t^{\infty} L_x^{2} }^2 \norm{u_{hi}^i u_{lo}^{5-i} \nabla u_{lo}}_{L_t^{1} L_x^{1} }  \lesssim \eta 2^{k_0} \parenthese{ 1+\norm{u_{\lambda}}_{\tilde{X}_{k_0}([0,\frac{T}{\lambda^2}]\times \R^2)}}^6  .
\end{equation*}

\item
Finally, take \eqref{eq P_5}. Apply H\"older, Sobolev embedding, Bernstein and Proposition \ref{prop X properties}
\begin{align*}
\norm{u_{hi}\nabla P_{lo} \mathcal{O}(u_{lo}^5)}_{L_t^{1} L_x^{1} } 
& \lesssim 2^{\frac{k_0}{2}}\norm{\abs{\nabla}^{\frac{1}{2}} u_{hi}}_{L_t^{\infty} L_x^{2}} \norm{\nabla u_{lo}}_{L_t^{4} L_x^{4} } \norm{u_{lo}}_{ L_t^{\frac{16}{3}} L_x^{16} }^4 
\lesssim \norm{u_{\lambda}}_{\tilde{X}_{k_0}([0,\frac{T}{\lambda^2}]\times \R^2)}^5 \\
\norm{u_{hi}\nabla P_{lo} \mathcal{O}(u_{hi}^5)}_{L_t^{1} L_x^{1} } 
& \lesssim \norm{u_{hi}}_{L_t^{4} L_x^{4} }  2^{-\frac{3k_0}{2}} \norm{ u_{hi}}_{L_t^{\infty} L_x^{4} }^3 \norm{ u_{hi}}_{L_t^{\frac{8}{3}} L_x^{8} }^2 
\lesssim \norm{u_{\lambda}}_{\tilde{X}_{k_0}([0,\frac{T}{\lambda^2}]\times \R^2)}^3  .
\end{align*}

\noindent Therefore,
\begin{equation*}
\eqref{eq P_5}  = \sum_{i=0}^5 \norm{ w}_{L_t^{\infty} L_x^{2} }^2  \norm{u_{hi}\nabla P_{lo} \mathcal{O}(u_{hi}^i u_{lo}^{5-i})}_{L_t^{1} L_x^{1} } 
\lesssim \eta 2^{k_0}  \parenthese{1+ \norm{u_{\lambda}}_{\tilde{X}_{k_0}([0,\frac{T}{\lambda^2}]\times \R^2)}}^5  .
\end{equation*}

\end{enumerate}

Hence, collect all the estimates, then we have
\begin{equation*}
\eqref{eq M_2} \lesssim \eta 2^{k_0} \parenthese{1+ \norm{u_{\lambda}}_{\tilde{X}_{k_0}([0,\frac{T}{\lambda^2}]\times \R^2)}}^7  .
\end{equation*}
\end{proof}

\begin{lem}\label{lem M_3}
There exists $\eta = \eta(K) > 0$ satisfying
\begin{equation*}
\eqref{eq M_3} \lesssim \eta^{\frac{1}{4}} 2^{k_0}  \parenthese{1+\norm{u_{\lambda}}_{\tilde{X}_{k_0}([0,\frac{T}{\lambda^2}]\times \R^2)}}^6 .
\end{equation*}
\end{lem}

\begin{proof}
By H\"older and Lemma \ref{lem H^1/2},
\begin{align*}
\eqref{eq M_3} & = \abs{ \int_{0}^{\frac{T}{\lambda^2}} \iint \im [\bar{w} \nabla w] (t,x) \frac{x-y}{\abs{x-y}} \im [ \bar{w}\NN] (t,y)  \, dxdy dt }\\
& \lesssim \norm{\abs{\nabla}^{\frac{1}{2}} w}_{L_t^{\infty} L_x^{2} }^2 \norm{\im [ \bar{w}\NN]}_{L_t^1 L_x^1 } \lesssim  \norm{\im [ \bar{w}\NN]}_{L_t^1 L_x^1}  .
\end{align*}

Then we reduce to estimating $\im [ \bar{w}\NN]$, that is,
\begin{equation*}
\im [ \bar{w}\NN] \lesssim \eta^{\frac{1}{4}} 2^{k_0} \parenthese{1+\norm{u_{\lambda}}_{\tilde{X}_{k_0}([0,\frac{T}{\lambda^2}]\times \R^2)}}^6 .
\end{equation*}

We first write
\begin{equation*}
\im [ \bar{w}\NN]= \mathcal{O}(u_{hi}^5 u_{lo}) + \mathcal{O}(u_{hi}^4 u_{lo}^2) + \mathcal{O}(u_{hi}^3 u_{lo}^3) + \mathcal{O}(u_{hi}^2 u_{lo}^4) + \mathcal{O}(u_{hi} P_{hi}(u_{lo}^5)) .
\end{equation*}

By H\"older, Sobolev embedding, Bernstein, \eqref{eq Interpolation} and Proposition \ref{prop X properties}
\begin{align*}
\norm{\mathcal{O}(u_{hi}^5 u_{lo})  }_{L_t^1 L_x^1 } & \lesssim \norm{u_{hi}}_{L_{t}^{\infty} L_{x}^{4} }^3 \norm{u_{hi}}_{L_{t}^{\frac{16}{7}} L_{x}^{16} }^2 \norm{u_{lo}}_{L_{t}^{8} L_{x}^{8} }  \lesssim \eta^{\frac{1}{16}} 2^{k_0} \norm{u_{\lambda}}_{\tilde{X}_{k_0}([0,\frac{T}{\lambda^2}]\times \R^2)}^{\frac{23}{8}} \\
\norm{\mathcal{O}(u_{hi} P_{hi} (u_{lo})^5)}_{L_{t}^1L_x^1} & 
\lesssim 2^{k_0}\norm{ u_{hi}}_{L_{t}^{4} L_{x}^{4} } \norm{\abs{\nabla}^{\frac{1}{2}} u_{lo}}_{L_{t}^{4} L_{x}^{4} }   \norm{ u_{lo}}_{L_{t}^{8} L_{x}^{8} }^4 \lesssim \eta^{\frac{1}{8}} 2^{k_0}\norm{u_{\lambda}}_{\tilde{X}_{k_0}([0,\frac{T}{\lambda^2}]\times \R^2)}^{\frac{23}{4}}   .
 \end{align*}
Therefore, 
\begin{equation*}
\norm{\im [\bar{w}\NN]}_{L_t^1 L_x^1} \lesssim \eta^{\frac{1}{4}} 2^{k_0}  \parenthese{1+\norm{u_{\lambda}}_{\tilde{X}_{k_0}([0,\frac{T}{\lambda^2}]\times \R^2)}}^6 .
\end{equation*}
\end{proof}

Then, combining Lemmas \ref{lem M_1}, \ref{lem M_2} and \ref{lem M_3} together, and by long time Strichartz estimates, we have 
\begin{equation*}
\norm{ \abs{\nabla}^{\frac{1}{2}} \abs{w(t,x)}^2 }_{L_t^2 L_x^2 ([0, \frac{T}{\lambda^2}] \times \R^2)}^2  \lesssim \eta^{\frac{1}{4}} 2^{k_0}  \parenthese{1+\norm{u_{\lambda}}_{\tilde{X}_{k_0}([0,\frac{T}{\lambda^2}]\times \R^2)}}^6 \lesssim \eta^{\frac{1}{4}} 2^{k_0}  .
\end{equation*}

Undoing the scaling $u(t,x) \mapsto \sqrt{\lambda} u( \lambda^2 t,\lambda x)$, $\lambda =\frac{K}{\varepsilon_3 2^{k_0}}$, we have
\begin{equation*}
\norm{ \abs{\nabla}^{\frac{1}{2}} \abs{P_{\geq \varepsilon_3 K^{-1}}u(t,x)}^2 }_{L_t^2 L_x^2 ([0, T] \times \R^2)}^2  \lesssim \varepsilon_3^{-1} \eta(K)^{\frac{1}{2}} K .
\end{equation*}
Now we complete the proof of Theorem \ref{thm FLIME}.
\end{proof}

\begin{rmk}
Realizing that Sobolev embedding gives us
\begin{equation*}
 \norm{u}_{L_t^4 L_x^8}^2  = \norm{\abs{u}^2}_{L_t^2 L_x^4}^2  \lesssim \norm{\abs{\nabla} ^{\frac{1}{2} } \abs{u}^2}_{L_t^2 L_x^2}^2 ,
\end{equation*}
hence Theorem \ref{thm FLIME} implies 
\begin{equation*}
\norm{ P_{ \geq \varepsilon_3 K^{-1}} u(t,x)}_{ L_t^4 L_x^8 ([0,T] \times \R^2)}^4  \lesssim \eta(K)^{\frac{1}{2}} K .
\end{equation*}

\end{rmk}

\subsection{Impossibility of quasi-soliton solutions}
We first state a concentration lemma:
\begin{lem}\label{lem Cmpt}
There is an $R_0= R_0(T)>0$
\begin{equation*}
\int_{\abs{x-x(t)} \leq \frac{R_0}{N(t)}}  \abs{P_{\geq \varepsilon_3 K^{-1}} u(t,x)}^4 dx  \gtrsim 1
\end{equation*}
uniformly for any $t \in [0,T]$
\end{lem}
\begin{proof}
The proof is followed from the proof of Lemma 4.2 in \cite{KM3}.
\end{proof}

%
%
%
%

\begin{thm}[Impossibility of quasi-soliton]
If $u$ is an almost periodic solution to \eqref{NLS} and $\int_0^{\infty} N(t) \, dt = \infty$, then $u \equiv 0$.
\end{thm}

\begin{proof}
Recall $K= \int_0^T N(t) \, dt$. By Lemma \ref{lem Cmpt}, the frequency-localized interaction Morawetz estimates and H\"older, we have that
\begin{align*}
1 
& \lesssim \lim_{k \nearrow \infty} \frac{1}{K} \int_0^T  N(t)  \parenthese{\int_{\abs{x-x(t)} \leq \frac{\eta(K)}{N(t)}}  \abs{P_{\geq \varepsilon_3 K^{-1}} u(t,x)}^4  \, dx  } dt \\
& \lesssim  \lim_{k \nearrow \infty}\frac{\eta(K)}{K}  \norm{P_{\geq \varepsilon_3 K^{-1} } u}_{L_t^4L_x^8 ([0,T]\times \R^2)}^4   \lesssim \lim_{K \nearrow \infty} \eta(K)^{3/2} = 0  .
\end{align*}
\noindent Therefore, $u \equiv 0$, contradiction.
\end{proof}

At this point, we have ruled out the existence of both finite-time blow-up solutions and quasi-soliton solutions, hence we complete the proof of Theorem \ref{thm Main}.

{\bf Acknowledgements} The author is very grateful to her advisor, Andrea R. Nahmod, for suggesting this problem and her patient guidance, encouragement and advice. The author would like to thank Benjamin Dodson for his helpful discussions and comments on a preliminary draft of this paper. The author acknowledges support from the National Science Foundation through her advisor Andrea R. Nahmod's grants \MakeUppercase{NSF}-\MakeUppercase{DMS} 1201443 and \MakeUppercase{NSF}-\MakeUppercase{DMS} 1463714.

\end{document}